\documentclass{article}

\usepackage{microtype}
\usepackage{graphicx}
\usepackage{subfigure}
\usepackage{booktabs} 

\usepackage{multirow, makecell}

\usepackage{hyperref}

\usepackage[accepted]{icml2021}

\usepackage[utf8]{inputenc}
\usepackage{amssymb}
\usepackage{amsthm,amsmath}
\usepackage{mathdots}

\makeatletter
\newtheorem*{rep@theorem}{\rep@title}
\newcommand{\newreptheorem}[2]{%
\newenvironment{rep#1}[1]{%
 \def\rep@title{#2 \ref{##1}}%
 \begin{rep@theorem}}%
 {\end{rep@theorem}}}
\makeatother

\newtheorem{theorem}{Theorem}
\newtheorem{corollary}{Corollary}
\newtheorem{lemma}{Lemma}

\newreptheorem{theorem}{Theorem}

\newcommand{\R}{\mathbb R}
\newcommand{\reals}{\mathbb R}
\newcommand{\spann}{\mathrm{span}}
\newcommand{\err}{\mathrm{Err}}
\newcommand{\lagrange}{\mathbf L}
\newcommand{\sop}{\mathbf G}

\newcommand{\bb}{\mathbf b}
\newcommand{\bc}{\mathbf c}
\newcommand{\be}{\mathbf e}
\newcommand{\bh}{\mathbf h}
\newcommand{\br}{\mathbf r}
\newcommand{\bs}{\mathbf s}

\newcommand{\bv}{\mathbf v}
\newcommand{\bw}{\mathbf w}
\newcommand{\bx}{\mathbf x}
\newcommand{\by}{\mathbf y}
\newcommand{\bz}{\mathbf z}

\newcommand{\bA}{\mathbf A}
\newcommand{\bB}{\mathbf B}
\newcommand{\bC}{\mathbf C}
\newcommand{\bG}{\mathbf G}
\newcommand{\bH}{\mathbf H}
\newcommand{\bI}{\mathbf I}
\newcommand{\bL}{\mathbf L}
\newcommand{\bN}{\mathbf N}
\newcommand{\bO}{\mathbf O}
\newcommand{\bQ}{\mathbf Q}
\newcommand{\bU}{\mathbf U}

\newcommand{\cA}{\mathcal A}
\newcommand{\cC}{\mathcal C}
\newcommand{\cE}{\mathcal E}
\newcommand{\cF}{\mathcal F}
\newcommand{\cJ}{\mathcal J}
\newcommand{\cK}{\mathcal K}
\newcommand{\cL}{\mathcal L}
\newcommand{\cN}{\mathcal N}
\newcommand{\cO}{\mathcal O}
\newcommand{\cP}{\mathcal P}
\newcommand{\fA}{\mathfrak A}

\newcommand{\even}{\mathrm{ev}}

\newcounter{relctr} 
\everydisplay\expandafter{\the\everydisplay\setcounter{relctr}{0}} 
\newcommand\labelrel[2]{%
  \begingroup
    \refstepcounter{relctr}%
    \stackrel{\textnormal{(\alph{relctr})}}{\mathstrut{#1}}%
    \originallabel{#2}%
  \endgroup
}
\AtBeginDocument{\let\originallabel\label}

\DeclareMathOperator{\Tr}{\mathrm{Tr}}

\icmltitlerunning{Accelerated \smash{$\mathcal{O}(1/k^2)$} Rate for Smooth Convex-Concave Minimax Problems on Squared Gradient Norm}

\begin{document}

\twocolumn[
\icmltitle{Accelerated Algorithms for Smooth Convex-Concave Minimax Problems \\with $\mathcal{O}(1/k^2)$ Rate on Squared Gradient Norm}



\icmlsetsymbol{equal}{*}

\begin{icmlauthorlist}
\icmlauthor{TaeHo Yoon}{SNU}
\icmlauthor{Ernest K. Ryu}{SNU}
\end{icmlauthorlist}

\icmlaffiliation{SNU}{Department of Mathematical Sciences, Seoul National University, Seoul, Korea}

\icmlcorrespondingauthor{Ernest K. Ryu}{ernestryu@snu.ac.kr}

\icmlkeywords{Convex optimization, Minimax optimization, Convex-Concave problems, Acceleration, Generative models, GANs}

\vskip 0.3in
]



\printAffiliationsAndNotice{}  

\begin{abstract}
In this work, we study the computational complexity of reducing the squared gradient magnitude for smooth minimax optimization problems.
First, we present algorithms with accelerated $\mathcal{O}(1/k^2)$ last-iterate rates, faster than the existing $\mathcal{O}(1/k)$ or slower rates for extragradient, Popov, and gradient descent with anchoring.
The acceleration mechanism combines extragradient steps with anchoring and is distinct from Nesterov's acceleration.
We then establish optimality of the $\mathcal{O}(1/k^2)$ rate through a matching lower bound.
\end{abstract}

\section{Introduction}
Minimax optimization problems, or minimax games, of the form
\begin{align}
    \label{SPP_min_max}
    \underset{\bx \in \reals^n}{\mbox{minimize}} \,\, \underset{\by \in \reals^m}{\mbox{maximize}} \,\, \lagrange(\bx,\by)
\end{align}
have recently gained significant interest in the optimization and machine learning communities due to their application in adversarial training \cite{goodfellow2014explaining, madry2017towards} and generative adversarial networks (GANs) \cite{goodfellow2014generative}.

Prior works on minimax optimization often consider compact domains $X, Y$ for $\bx, \by$ and use the \emph{duality gap}
\begin{align*}
    \err_{\mathrm{gap}}(\bx,\by) := \sup_{\Tilde{\by} \in Y}\, \lagrange(\bx,\Tilde{\by}) - \inf_{\Tilde{\bx} \in X}\, \lagrange(\Tilde{\bx},\by)
\end{align*}
to quantify suboptimality of algorithms' iterates in solving \eqref{SPP_min_max}.
However, while it is a natural analog of minimization error for minimax problems, the duality gap can be difficult to measure directly in practice, and it is unclear how to generalize the notion to non-convex-concave problems.

In contrast, the squared gradient magnitude $\|\nabla \lagrange(\bx,\by)\|^2$, when $\lagrange$ is differentiable, is a more directly observable value for quantifying suboptimality.
Moreover, the notion is meaningful for differentiable non-convex-concave minimax games.
Interestingly, very few prior works have analyzed convergence rates on the gradient norm for minimax problems, and the optimal convergence rate or corresponding algorithms were hitherto unknown.

\paragraph{Contributions.}
In this work, we introduce the \emph{extra anchored gradient (EAG)} algorithms for smooth convex-concave minimax problems and establish an accelerated $\|\nabla \lagrange(\bz^k)\|^2 \leq \mathcal O(R^2/k^2)$ rate, where $R$ is the Lipschitz constant of $\nabla \lagrange$.
The rate improves upon the $\mathcal{O}(R^2/k)$ rates of prior algorithms and is the first $\mathcal{O}(R^2/k^2)$ rate in this setup.
We then provide a matching $\Omega(R^2/k^2)$ complexity lower bound for gradient-based algorithms and thereby establish optimality of EAG.

Beyond establishing the optimal complexity, our results provide the following observations.
First, different suboptimality measures lead to materially different acceleration mechanisms, since reducing the duality gap is done optimally by the extragradient algorithm \citep{nemirovski2004prox, nemirovsky1992information}.
Also, since our optimal accelerated convergence rate is on the non-ergodic last iterate, neither averaging nor keeping track of the best iterate is necessary for optimally reducing the gradient magnitude in the deterministic setup.

\subsection{Preliminaries and notation}
We say a saddle function $\lagrange\colon \reals^n \times \reals^m \to \reals$ is convex-concave if $\lagrange(\bx,\by)$ is convex in $\bx\in\reals^n$ for all fixed $\by\in \reals^m$ and $\lagrange(\bx,\by)$ is concave in $\by\in\reals^m$ for all fixed $\bx\in \reals^n$.
We say $(\bx^\star, \by^\star)$ is a saddle point of $\lagrange$ if $\lagrange(\bx^\star,\by) \le \lagrange(\bx^\star,\by^\star) \le \lagrange(\bx,\by^\star)$
for all $\bx\in \reals^n$ and $\by\in \reals^m$.
Solutions to the minimax problem \eqref{SPP_min_max} are defined to be saddle points of $\lagrange$.
For notational conciseness, write $\bz = (\bx,\by)$.
When $\lagrange$ is differentiable, define the \emph{saddle operator} of $\lagrange$ at $\bz = (\bx,\by)$ by
\begin{align}
    \label{eqn:saddle_subdifferential}
    \sop_\lagrange (\bz) = \begin{bmatrix}
    \nabla_\bx \lagrange(\bx,\by)\\
    - \nabla_\by \lagrange(\bx,\by)
    \end{bmatrix}.
\end{align}
(When clear from the context, we drop the subscript $\lagrange$.)
The saddle operator is \emph{monotone} \cite{rockafellar1970monotone}, i.e., $\langle \sop(\bz_1) - \sop(\bz_2), \bz_1 - \bz_2 \rangle \ge 0$ for all $\bz_1,\bz_2 \in \reals^{n}\times\reals^m$.
We say $\lagrange$ is $R$-smooth if $\sop_\lagrange$ is $R$-Lipschitz continuous.
Note that $\nabla \lagrange\ne \sop_\lagrange$ due to the sign change in the $\by$ gradient, but $\|\nabla \lagrange\| = \|\sop_\lagrange\|$, and we use the two forms interchangeably.
Because $\bz^\star = (\bx^\star,\by^\star)$ is a saddle point of $\lagrange$ if and only if $0 =\sop_\lagrange(\bz^\star)$, 
the squared gradient magnitude is a natural measure of suboptimality at a given point for smooth convex-concave problems.

\subsection{Prior work}

\paragraph{Extragradient-type algorithms.}
The first main component of our proposed algorithm is the extragradient (EG) algorithm of \citet{korpelevich1977extragradient}.
EG and its variants, including the algorithm of \citet{popov1980modification}, have been studied in the context of saddle point and variational inequality problems and have appeared in the mathematical programming literature \cite{solodov1999hybrid, tseng2000modified, noor2003new, censor2011subgradient, lyashko2011low, malitsky2014extragradient, malitsky2015projected, malitsky2019golden}.
More recently in the machine learning literature, similar ideas such as optimism \cite{chiang2012online, rakhlin2013online}, prediction \cite{yadav2017stabilizing}, and negative momentum \cite{gidel2019negative, zhang2020unified}
have been presented and used in the context of multi-player games \cite{daskalakis2011near, rakhlin2013optimization, syrgkanis2015fast, antonakopoulos2020adaptive} and GANs \cite{gidel2018variational, mertikopoulos2019optimistic, liang2019interaction, peng2020training}.

\paragraph{$\boldsymbol{\mathcal{O}(R/k)}$ rates on duality gap.}
For minimax problems with an $R$-smooth $\lagrange$ and bounded domains for $\bx$ and $\by$,
\citet{nemirovski2004prox} presented the mirror-prox algorithm generalizing EG and established ergodic $\mathcal O(R/k)$ convergence rates on $\err_\mathrm{gap}$.
\citet{nesterov2007dual, monteiro2010complexity, monteiro2011complexity} extended the $\mathcal O(R/k)$ complexity analysis to the case of unbounded domains.
\citet{mokhtari2020convergence} showed that the optimistic descent converges at $\mathcal{O}(R/k)$ rate with respect to $\err_\mathrm{gap}$.
Since there exists $\Omega (R/k)$ complexity lower bound on $\err_{\mathrm{gap}}$ for black-box gradient-based minimax optimization algorithms \citep{nemirovsky1992information, nemirovski2004prox}, in terms of duality gap, these algorithms are order-optimal.

\paragraph{Convergence rates on squared gradient norm.}
Using standard arguments (e.g.\ \citep[Lemma~2.3]{solodov1999hybrid}), one can show $\underset{i=0,\dots,k}{\min}\, \|\sop(\bz^i)\|^2 \leq \mathcal O(R^2/k)$ convergence rate of EG, provided that $\lagrange$ is $R$-smooth.
\citet{ryu2019ode} showed that optimistic descent algorithms also attain $\mathcal O(R^2/k)$ convergence in terms of the best iterate and
proposed simultaneous gradient descent with \emph{anchoring}, which pulls iterates toward the initial point $\bz^0$, and established $\mathcal O(R^2/k^{2-2p})$ convergence rates in terms of squared gradient norm of the last iterate (where $p>\frac{1}{2}$ is an algorithm parameter; see Section \ref{section:algorithm_specifications}).
Notably, anchoring resembles the Halpern iteration \citep{halpern1967fixed, lieder2020convergence}, which was used in \citet{diakonikolas2020halpern} to develop a regularization-based algorithm with near-optimal (optimal up to logarithmic factors) complexity with respect to the gradient norm of the last iterate.
Anchoring turns out to be the second main component of the acceleration; combining EG steps with anchoring, we obtain the optimal last-iterate convergence rate of $\mathcal O(R^2/k^2)$.

\paragraph{Structured minimax problems.}
For structured minimax problems of the form
\begin{align*}
    \lagrange(\bx,\by) = f(\bx) + \langle \bA\bx, \by \rangle - g(\by),
\end{align*}
where $f,g$ are convex and $\bA$ is a linear operator, primal-dual splitting algorithms \citep{chambolle2011first,condat2013primal,vu2013splitting,yan2018new,ryu2020LSCOMO} and Nesterov's smoothing technique \cite{nesterov2005excessive,nesterov2005smooth} have also been extensively studied \cite{chen2014optimal,he2016accelerated}.
Notably, when $g$ is of ``simple'' form,
Neterov's smoothing framework achieves an accelerated rate $\mathcal O\left(\frac{\|\bA\|}{k} + \frac{L_f}{k^2}\right)$ on duality gap.
Additionally, \citet{chambolle2016ergodic} have shown that splitting algorithms can achieve $\mathcal O(1/k^2)$ or linear convergence rates under appropriate strong convexity and smoothness assumptions on $f$ and $g$, although they rely on proximal operations.
\citet{kolossoski2017accelerated, hamedani2018primal, zhao2019optimal, alkousa2020accelerated} generalized these accelerated algorithms to the setting where the coupling term $\langle \bA\bx, \by \rangle$ is replaced by non-bilinear convex-concave function $\Phi (\bx,\by)$.

\paragraph{Complexity lower bounds.}
\citet{ouyang2019lower} presented a $\Omega\left(\frac{\|\bA\|}{k} + \frac{L_f}{k^2}\right)$ complexity lower bound on duality gap for gradient-based algorithms solving bilinear minimax problems with proximable $g$, establishing optimality of Nesterov's smoothing.
\citet{zhang2019lower} presented lower bounds for strongly-convex-strongly-concave problems.
\citet{golowich2020last} proved that with the narrower class of \emph{$1$-SCLI} algorithms, which includes EG but not EAG, the squared gradient norm of the last iterate cannot be reduced beyond $\mathcal{O}(R^2/k)$ in $R$-smooth minimax problems.
These approaches are aligned with the information-based complexity analysis, introduced in \cite{nemirovsky1983problem} and thoroughly studied in \cite{nemirovsky1991optimality, nemirovsky1992information} for the special case of linear equations.

\paragraph{Other problem setups.}
\citet{nesterov2009primal} and \citet{nedic2009subgradient} proposed subgradient algorithms for non-smooth minimax problems.
Stochastic minimax and variational inequality problems were studied in \citep{nemirovski2009robust, juditsky2011solving, lan2012optimal, ghadimi2012optimal, ghadimi2013optimal, chen2014optimal, chen2017accelerated, hsieh2019convergence}.
Strongly monotone variational inequality problems or strongly-convex-strongly-concave minimax problems were studied in \citep{tseng1995linear, nesterov2006solving, gidel2018variational, mokhtari2020unified, lin2020near, wang2020improved, zhang2020unified, azizian2020tight}.
Recently, minimax problems with objectives that are either strongly convex or nonconvex in one variable were studied in \citep{rafique2018non, thekumparampil2019efficient, jin2019minmax, nouiehed2019solving, ostrovskii2020efficient, lin2020gradient, lin2020near, lu2020hybrid, wang2020improved, yang2020catalyst, chen2021proximal}.
Minimax optimization of composite objectives with smooth and nonsmooth-but-proximable convex-concave functions were studied in \citep{tseng2000modified, csetnek2019shadow, malitsky2020forward, bui2021warped}.


\section{Accelerated algorithms: Extra anchored gradient}
\label{section:Accelerated_algorithm}
We now present two accelerated EAG algorithms that are qualitatively very similar but differ in the choice of step-sizes.
The two algorithms present a tradeoff between the simplicity of the step-size and the simplicity of the convergence proof; one algorithm has a varying step-size but a simpler convergence proof, while the other algorithm has a simpler constant step-size but has a more complicated proof.


\subsection{Description of the algorithms}
\label{section:algorithms}
The proposed extra anchored gradient (EAG) algorithms have the following general form:
\begin{equation}
\label{eqn:EAG-general}
\begin{aligned}
\bz^{k+1/2}&=\bz^k +\beta_k(\bz^0-\bz^k)- \alpha_k \sop(\bz^k)\\
\bz^{k+1}&=\bz^k +\beta_k(\bz^0-\bz^k)-\alpha_k \sop(\bz^{k+1/2})
\end{aligned}
\end{equation}
for $k\ge 0$, where $\bz^0\in\reals^n\times\reals^m$ is the starting point.
We use $\sop$ defined in \eqref{eqn:saddle_subdifferential} rather than describing the $\bx$- and $\by$- updates separately to keep the notation concise.
We call $\alpha_k > 0$ \emph{step-sizes} and $\beta_k \in [0,1)$ \emph{anchoring coefficients}.
Note that when $\beta_k = 0$, EAG coincides with the unconstrained extragradient algorithm.

The simplest choice of $\{\alpha_k\}_{k\ge 0}$ is the constant one.
Together with the choice $\beta_k = \frac{1}{k+2}$ (which we clarify later), we get the following simpler algorithm.

\paragraph{EAG with constant step-size (EAG-C)}
\begin{align*}
\bz^{k+1/2}&=\bz^k+\frac{1}{k+2}(\bz^0-\bz^k)-\alpha \sop(\bz^k)\\
\bz^{k+1}&=\bz^k+\frac{1}{k+2}(\bz^0-\bz^k)-\alpha \sop(\bz^{k+1/2})
\end{align*}
where $\alpha > 0$ is fixed.

\begin{theorem}
\label{thm:EAG-C}
Assume $\lagrange\colon \reals^n \times \reals^m \to \reals$ is an $R$-smooth convex-concave function with a saddle point $\bz^\star$. 
Assume $\alpha > 0$ satisfies
\begin{equation}
\label{eqn:EAG-C-alpha-restriction}
\begin{aligned}
    & 1 - 3\alpha R - \alpha^2 R^2 - \alpha^3 R^3 \geq 0\\
    & 1 - 8\alpha R + \alpha^2 R^2 - 2\alpha^3 R^3 \geq 0.
\end{aligned}
\end{equation}
Then EAG-C converges with rate
\[
\|\nabla \lagrange(\bz^k)\|^2\le \frac{4(1+\alpha R+\alpha^2 R^2)}{\alpha^2 (1+\alpha R)} \frac{\|\bz^0-\bz^\star\|^2}{(k+1)^2}
\]
for $k \ge 0$.
\end{theorem}

\begin{corollary}
\label{coro:EAG-C}
In the setup of Theorem~\ref{thm:EAG-C}, $\alpha \in \left( 0,\frac{1}{8R} \right]$ satisfies \eqref{eqn:EAG-C-alpha-restriction}, and the particular choice $\alpha = \frac{1}{8R}$ yields
\[
\|\nabla \lagrange(\bz^k)\|^2\le \frac{260 R^2\|\bz^0-\bz^\star\|^2}{(k+1)^2}
\]
for $k \ge 0$.
\end{corollary}

While EAG-C is simple in its form, its convergence proof (presented in the appendix) is complicated.
Furthermore, the constant $260$ in Corollary~\ref{coro:EAG-C} seems large and raises the question of whether it could be reduced.
These issues, to some extent, are addressed by the following alternative version of EAG.

\paragraph{EAG with varying step-size (EAG-V)}
\begin{align*}
\bz^{k+1/2}&=\bz^k+\frac{1}{k+2}(\bz^0-\bz^k)-\alpha_k \sop(\bz^k)\\
\bz^{k+1}&=\bz^k+\frac{1}{k+2}(\bz^0-\bz^k)-\alpha_k \sop(\bz^{k+1/2}),
\end{align*}
where $\alpha_0 \in \left(0, \frac{1}{R}\right)$ and
\begin{align}
    \label{eqn:recurrence-alpha}
    \alpha_{k+1} &= \frac{\alpha_k}{1-\alpha_k^2 R^2} \left( 1 - \frac{(k+2)^2}{(k+1)(k+3)}\alpha_k^2 R^2 \right) \nonumber \\
    &= \alpha_k \left( 1 - \frac{1}{(k+1)(k+3)} \frac{\alpha_k^2 R^2}{1-\alpha_k^2 R^2} \right)
\end{align}
for $k \geq 0$.

As the recurrence relation \eqref{eqn:recurrence-alpha} may seem unfamiliar, we provide the following lemma describing the behavior of the resulting sequence.

\begin{lemma}
\label{lemma:alpha}
If $\alpha_0 \in \left(0,\frac{3}{4R}\right)$, then the sequence $\{\alpha_k\}_{k\geq 0}$ of \eqref{eqn:recurrence-alpha} monotonically decreases to a positive limit.
In particular, when $\alpha_0 = \frac{0.618}{R}$, we have $\lim_{k\rightarrow\infty}\alpha_k\approx \frac{0.437}{R}$.
\end{lemma}

We now state the convergence results for EAG-V.

\begin{theorem}
\label{thm:EAG-V}
Assume $\lagrange\colon \reals^n \times \reals^m \to \reals$ is an $R$-smooth convex-concave function with a saddle point $\bz^\star$.
Assume $\alpha_0 \in \left(0,\frac{3}{4R}\right)$, and define $\alpha_\infty = \lim_{k\to\infty} \alpha_k$.
Then EAG-V converges with rate
\[
\|\nabla \lagrange(\bz^k)\|^2\leq \frac{4\left( 1+\alpha_0\alpha_\infty R^2 \right)}{\alpha_\infty^2} \frac{\|\bz^0-\bz^\star\|^2}{(k+1)(k+2)}
\]
for $k \ge 0$.
\end{theorem}

\begin{corollary}
\label{coro:EAG-V}
EAG-V with $\alpha_0 = \frac{0.618}{R}$ satisfies
\[
\|\nabla \lagrange(\bz^k)\|^2\leq \frac{27 R^2\|\bz^0-\bz^\star\|^2}{(k+1)(k+2)}
\]
for $k \ge 0$.
\end{corollary}

\subsection{Proof outline}
\label{section:conv-proof-outline}
We now outline the convergence analysis for EAG-V, whose proof is simpler than that of EAG-C.
The key ingredient of the proof is a Lyapunov analysis with a nonincreasing Lyapunov function, the $V_k$ of the following lemma.
 
\begin{lemma}
\label{lemma:lyap}
Let $\{\beta_k\}_{k\geq 0} \subseteq (0,1)$ and $\alpha_0 \in \left(0, \frac{1}{R}\right)$ be given. 
Define the sequences $\{A_k\}_{k\geq 0}, \{B_k\}_{k\geq 0}$ and $\{\alpha_k\}_{\geq 0}$ by the recurrence relations
\begin{align}
    & A_k = \frac{\alpha_k}{2\beta_k} B_k \label{eqn:lemma-recur-ak}\\
    & B_{k+1} = \frac{B_k}{1-\beta_k} \label{eqn:lemma-recur-bk}\\
    & \alpha_{k+1} = \frac{\alpha_k \beta_{k+1} (1-\alpha_k^2 R^2-\beta_k^2)}{\beta_k (1-\beta_k)(1-\alpha_k^2 R^2)} \label{eqn:lemma-recur-alphak}  
\end{align}
for $k\ge 0$, where $B_0 = 1$.
Suppose that $\alpha_k \in (0,\frac{1}{R})$ holds for all $k\geq 0$.
Assume $\lagrange$ is $R$-smooth and convex-concave.
Then the sequence $\{V_k\}_{k\ge 0}$ defined as
\begin{align}
    \label{eqn:lyap}
    V_k := A_k \|\sop(\bz^k)\|^2 + B_k \langle \sop (\bz^k), \bz^k - \bz^0 \rangle    
\end{align}
for EAG iterations in \eqref{eqn:EAG-general} is nonincreasing.
\end{lemma}

In Lemma~\ref{lemma:lyap}, the choice of $\beta_k = \frac{1}{k+2}$ leads to $B_k = k+1$, $A_k = \frac{\alpha_k (k+2)(k+1)}{2}$, and \eqref{eqn:recurrence-alpha}.
Why the Lyapunov function of Lemma~\ref{lemma:lyap} leads to the convergence guarantee of Theorem~\ref{thm:EAG-V} may not be immediately obvious.
The following proof provides the analysis.

\begin{proof}[Proof of Theorem \ref{thm:EAG-V}]
Let $\beta_k = \frac{1}{k+2}$ as specified by the definition of EAG-V.
By Lemma~\ref{lemma:lyap}, the quantity $V_k$ defined by \eqref{eqn:lyap} is nonincreasing in $k$.
Therefore,
\begin{align*}
    V_k \leq \cdots  \leq V_0 = \alpha_0 \|\sop(\bz^0)\|^2 \leq \alpha_0 R^2 \|\bz^0 - \bz^\star\|^2.
\end{align*}
Next, we have
\begin{alignat*}{3}
    & V_k & & = \, &  & \,A_k \| \sop(\bz^k) \|^2 + B_k \langle \sop(\bz^k), \bz^k-\bz^0 \rangle \nonumber\\
    & & & \labelrel\geq{ineq:proof_mono} \, & & \, A_k \| \sop(\bz^k) \|^2 + B_k \langle \sop(\bz^k), \bz^\star-\bz^0 \rangle \\
    & & & \labelrel\geq{ineq:proof_young} \, & & \, A_k\|\sop(\bz^k)\|^2 - \frac{A_k}{2} \|\sop(\bz^k)\|^2 - \frac{B_k^2}{2A_k} \|\bz^0-\bz^\star\|^2 \nonumber\\
    & & & \labelrel={eqn:akbk_substitution} \, & & \, \frac{\alpha_k}{4} (k+1)(k+2)\|\sop(\bz^k)\|^2 \\
    & & & & & \quad - \frac{k+1}{\alpha_k (k+2)} \|\bz^0-\bz^\star\|^2 \nonumber\\
    & & & \labelrel\geq{ineq:proof_lemma} \, & & \, \frac{\alpha_\infty}{4} (k+1)(k+2)\|\sop(\bz^k)\|^2 - \frac{1}{\alpha_\infty} \|\bz^0-\bz^\star\|^2,
\end{alignat*}
where \eqref{ineq:proof_mono} follows from the monotonicity inequality $\langle \sop(\bz^k), \bz^k - \bz^\star \rangle  \geq 0$,
\eqref{ineq:proof_young} follows from Young's inequality,
\eqref{eqn:akbk_substitution} follows from plugging in $A_k = \frac{\alpha_k (k+1)(k+2)}{2}$ and $B_k = k+1$, 
and \eqref{ineq:proof_lemma} follows from Lemma~\ref{lemma:alpha} ($\alpha_k \downarrow \alpha_\infty$).
Reorganize to get
\begin{align*}
    \frac{\alpha_\infty}{4}(k+1)(k+2) \|\sop&(\bz^k)\|^2 \le V_k + \frac{1}{\alpha_\infty} \|\bz^0 - \bz^\star\|^2 \\
    & \le \left( \alpha_0 R^2 + \frac{1}{\alpha_\infty} \right) \|\bz^0 - \bz^\star\|^2,
\end{align*}
and divide both sides by $\frac{\alpha_\infty}{4}(k+1)(k+2)$.
\end{proof}

\subsection{Discussion of further generalizations}
The algorithms and results of Sections~\ref{section:algorithms} and \ref{section:conv-proof-outline} remain valid when we replace $\sop$ with an $R$-Lipschitz continuous monotone operator;
neither the definition of the EAG algorithms nor any part of the proofs of Theorems~\ref{thm:EAG-C} and \ref{thm:EAG-V} utilize properties of saddle functions beyond the monotonicity of their subdifferentials.

For EAG-C, the step-size conditions \eqref{eqn:EAG-C-alpha-restriction} in Theorem \ref{thm:EAG-C} can be relaxed to accommodate larger values of $\alpha$.
However, we do not pursue such generalizations to keep the already complicated and arduous analysis of EAG-C manageable. Also, larger step-sizes are more naturally allowed in EAG-V and Theorem~\ref{thm:EAG-V}.
Finally, although \eqref{eqn:EAG-C-alpha-restriction} holds for values of $\alpha$ up to $\frac{0.1265}{R}$, we present a slightly smaller range $\left( 0,\frac{1}{8R} \right]$ in Corollary \ref{coro:EAG-C} for simplicity.

For EAG-V, the choice $\beta_k=\frac{1}{k+2}$ was obtained by roughly, but not fully, optimizing the bound on EAG-V originating from Lemma~\ref{lemma:lyap}.
If one chooses $\beta_k = \frac{1}{k+\delta}$ with $\delta > 1$, then \eqref{eqn:lemma-recur-ak} and \eqref{eqn:lemma-recur-bk} become
\begin{align*}
    \label{eqn:akbk}
    A_k = \frac{\alpha_k (k+\delta)(k+\delta-1)}{2(\delta-1)}, \quad B_k = \frac{k+\delta-1}{\delta-1}.
\end{align*}
As the proof of Theorem \ref{thm:EAG-V} illustrates, linear growth of $B_k$ and quadratic growth of $A_k$ leads to $\mathcal{O}(1/k^2)$ convergence of $\|\sop (\bz^k) \|^2$.
The value $\alpha_0 = \frac{0.618}{R} $ in Lemma~\ref{lemma:alpha} and Corollary~\ref{coro:EAG-V} was obtained by numerically minimizing the constant $\frac{4}{\alpha_\infty^2}\left(1+\alpha_0 \alpha_\infty R^2\right)$ in Theorem \ref{thm:EAG-V} in the case of $\delta=2$.
The choice $\delta = 2$, however, is not optimal.
Indeed, the constant $27$ of Corollary~\ref{coro:EAG-V} can be reduced to $24.44$ with $(\delta^\star, \alpha_0^\star) \approx (2.697, 0.690/R)$, which was obtained by numerically optimizing over $\delta$ and $\alpha_0$.
Finally, there is a possibility that a choice of $\beta_k$ not in the form of $\beta_k=\frac{1}{k+\delta}$ leads to an improved constant.

In the end, we choose to present EAG-C and EAG-V with the simple choice $\beta_k=\frac{1}{k+2}$. As we establish in Section~\ref{sec:lower-bound}, the EAG algorithms are optimal up to a constant.

\section{Optimality of EAG via a matching complexity lower bound}
\label{sec:lower-bound}
Upon seeing an accelerated algorithm, it is natural to ask whether the algorithm is optimal.
In this section, we present a $\Omega(R^2/k^2)$ complexity lower bound for the class of deterministic gradient-based algorithms for smooth convex-concave minimax problems. This result establishes that EAG is indeed optimal.

For the class of smooth minimax optimization problems, a deterministic \emph{algorithm} $\mathcal{A}$ produces iterates $(\bx^k, \by^k) = \bz^k$ for $k\ge 1$ given a starting point $(\bx^0, \by^0) = \bz^0$ and a saddle function $\lagrange$, and we write $\bz^k = \mathcal{A}(\bz^0, \dots, \bz^{k-1}; \lagrange)$  for $k\ge 1$.
Define $\mathfrak{A}_{\textrm{sim}}$ as the class of algorithms satisfying
\begin{equation}
    \bz^k \in \bz^0 + \spann \{ \sop_\lagrange (\bz^0), \dots, \sop_\lagrange (\bz^{k-1}) \},
    \label{eqn:alg_span_sim}
\end{equation}
and $\mathfrak{A}_{\textrm{sep}}$  as the class of algorithms satisfying
\begin{align}
    & \bx^{k} \in \bx^0 + \spann \left \{\nabla_\bx \lagrange(\bx^{0},\by^{0}), \dots, \nabla_\bx \lagrange(\bx^{k-1},\by^{k-1}) \right\} \nonumber \\
    & \by^{k} \in \by^0 + \spann \left\{\nabla_\by \lagrange(\bx^{0},\by^{0}), \dots, \nabla_\by \lagrange(\bx^{k-1},\by^{k-1}) \right\}.
    \label{eqn:alg-span}
\end{align}
To clarify, algorithms in $\mathfrak{A}_{\textrm{sim}}$ access and utilize the $\bx$- and $\by$-subgradients \emph{simultaneously}.
So $\mathfrak{A}_{\textrm{sim}}$ contains simultaneous gradient descent, extragradient, Popov, and EAG (if we also count intermediate sequences $\bz^{k+1/2}$ as algorithms' iterates).
On the other hand, algorithms in $\mathfrak{A}_{\textrm{sep}}$ can access and utilize the $\bx$- and $\by$-subgradients \emph{separately}.
So $\mathfrak{A}_{\textrm{sim}}\subset \mathfrak{A}_{\textrm{sep}}$, and alternating gradient descent-ascent belongs to $\fA_{\textrm{sep}}$ but not to $\fA_{\textrm{sim}}$.
    
In this section, we present a complexity lower bound that applies to all algorithms in $\mathfrak{A}_{\textrm{sep}}$, not just the algorithms in  $\mathfrak{A}_{\textrm{sim}}$.
Although EAG-C and EAG-V are in $\mathfrak{A}_{\textrm{sim}}$, we consider the broader class $\mathfrak{A}_{\textrm{sep}}$ to rule out the possibility that separately updating the $\bx$- and $\by$-variables provides an improvement beyond a constant factor.

We say $\lagrange(\bx,\by)$ is biaffine if it is an affine function of $\bx$ for any fixed $\by$ and an affine function of $\by$ for any fixed $\bx$.
Biaffine functions are, of course, convex-concave.
We first establish a complexity lower bound on minimiax optimization problems with biaffine loss functions.

\begin{theorem}
\label{thm:lowerbound}
Let $k \ge 0$ be fixed.
For any $n \ge k+2$, there exists an $R$-smooth biaffine function $\lagrange$ on $\reals^n \times \reals^n$ for which
\begin{align}
\label{eqn:thm_biaffine_lowerbound}
\|\nabla\lagrange(\bz^k)\|^2 \ge \frac{R^2\|\bz^0-\bz^\star\|^2}{(2\lfloor k/2 \rfloor+1)^2}    
\end{align}
holds for any algorithm in $\fA_{\textrm{sep}}$, where $\left\lfloor \cdot \right\rfloor$ is the floor function and $\bz^\star$ is the saddle point of $\lagrange$ closest to $\bz^0$.
Moreover, this lower bound is optimal in the sense that it cannot be improved with biaffine functions.
\end{theorem}

Since smooth biaffine functions are special cases of smooth convex-concave functions, Theorem~\ref{thm:lowerbound} implies the optimality of EAG applied to smooth convex-concave mimimax optimization problems.
\begin{corollary}
\label{cor:EAG-opt}
For $R$-smooth convex-concave minimax problems, an algorithm in $\fA_{\textrm{sep}}$ cannot attain a worst-case convergence rate better than
\[
\frac{R^2\|\bz^0-\bz^\star\|^2}{(2\lfloor k/2 \rfloor+1)^2}
\]
with respect to $\|\nabla\lagrange(\bz^k)\|^2$.
Since EAG-C and EAG-V have rates $\mathcal{O}(R^2\|\bz^0-\bz^\star\|^2/k^2)$, they are optimal, up to a constant factor, in $\fA_{\textrm{sep}}$.
\end{corollary}

\subsection{Outline of the worst-case biaffine construction}
\label{ss:biaffine-worst-case}
Consider biaffine functions of the form
\begin{align*}
    \lagrange (\bx,\by) = \langle \bA\bx-\bb, \by-\bc \rangle,
\end{align*}
where $\bA \in \R^{n\times n}$ and $\bb,\bc\in \R^n$.
Then, $\nabla_\bx \lagrange (\bx,\by) = \bA^\intercal (\by-\bc)$, $\nabla_\by \lagrange (\bx,\by) = \bA\bx-\bb$, $\sop$ is $\|\bA\|$-Lipschitz, and solutions to
\begin{align*}
    \underset{\bx \in X}{\mbox{minimize}} \,\, \underset{\by \in Y}{\mbox{maximize}} \,\, \langle \bA\bx-\bb, \by-\bc \rangle
\end{align*}
are characterized by $\bA\bx-\bb = 0$ and $ \bA^\intercal (\by-\bc) = 0$.

Through translation, we may assume without loss of generality that $\bx^0 = 0, \by^0 = 0$.
In this case, \eqref{eqn:alg-span} becomes
\begin{align}
    \bx^k \in & \,\, \spann\{ \bA^\intercal \bc, \bA^\intercal(\bA \bA^\intercal) \bc, \dots, \bA^\intercal (\bA \bA^\intercal)^{\lfloor \frac{k-1}{2} \rfloor} \bc \} \nonumber \\
    & + \spann\{ \bA^\intercal \bb, \bA^\intercal (\bA\bA^\intercal)\bb, \dots, \bA^\intercal (\bA\bA^\intercal)^{\lfloor \frac{k}{2} \rfloor - 1} \bb \} \nonumber \\
    \by^k \in & \,\, \spann\{ b, (\bA\bA^\intercal)\bb, \dots, (\bA\bA^\intercal)^{\lfloor \frac{k-1}{2} \rfloor} \bb \} \nonumber \\
    & + \spann\{ \bA\bA^\intercal \bc, \dots, (\bA\bA^\intercal)^{\lfloor \frac{k}{2} \rfloor} \bc \} \label{eqn:alg_span_expanded}
\end{align}
for $k\ge 2$. (We detail these arguments in the appendix.)
Furthermore let $\bA=\bA^\intercal$ and $\bb = \bA^\intercal \bc = \bA\bc$. Then the characterization of $ \fA_{\textrm{sep}}$ further simplifies to
\begin{align*}
    \bx^{k}, \by^{k} \in \mathcal{K}_{k-1} (\bA;\bb) := \spann \{ \bb, \bA\bb, \bA^2 \bb, \dots, \bA^{k-1}\bb \} .
\end{align*}
Note that $\mathcal{K}_{k-1} (\bA;\bb)$ is the order-$(k-1)$ Krylov subspace.

Consider the following lemma. Its proof, deferred to the appendix, combines arguments from \citet{nemirovsky1991optimality, nemirovsky1992information}.
\begin{lemma} 
\label{lemma:Nemirovsky}
Let $R>0$, $k \ge 0$, and $n\geq k+2$.
Then there exists $\bA=\bA^\intercal\in \reals^{n\times n}$ such that $\|\bA\|\le R$ and $\bb\in\mathcal{R}(\bA)$, satisfying
\begin{align}
\label{eqn:lemma_lowerbound}
\|\bA\bx - \bb\|^2 \geq \frac{R^2 \|\bx^\star\|^2}{(2\lfloor k/2 \rfloor+1)^2}
\end{align}
for any $\bx \in \mathcal{K}_{k-1}(\bA;\bb)$, where $\bx^\star$ is the minimum norm solution to the equation $\bA\bx=\bb$.
\end{lemma}
Take $\bA$ and $\bb$ as in Lemma~\ref{lemma:Nemirovsky} and $\bc = \bx^\star$.
Then $\bz^\star = (\bx^\star, \bx^\star)$ is the saddle point of $\lagrange(\bx,\by) = \langle \bA\bx-\bb, \by-\bc \rangle$ with minimum norm.
Finally,
\begin{align*}
    \|\nabla \lagrange(\bx^k, \by^k)\|^2 &= \|\bA^\intercal (\by^k - \bc)\|^2 + \|\bA\bx^k - \bb\|^2\\
    &= \|\bA\by^k - \bb\|^2 + \|\bA\bx^k - \bb\|^2 \\
    &\geq \frac{R^2 \|\bx^\star\|^2}{(2\lfloor k/2 \rfloor+1)^2} + \frac{R^2 \|\bx^\star\|^2}{(2\lfloor k/2 \rfloor+1)^2} \\
    &= \frac{R^2 \|\bz^\star - \bz^0\|^2}{(2\lfloor k/2 \rfloor+1)^2},
\end{align*}
for any $\bx^{k}, \by^{k} \in \mathcal{K}_{k-1} (\bA;\bb)$.
This completes the construction of the biaffine $\lagrange$ of Theorem~\ref{thm:lowerbound}.

\subsection{Optimal complexity lower bound}
\label{ss:optimal-complexity}
We now formalize the notion of complexity lower bounds.
This formulation will allow us to precisely state and prove the second statement of Theorem~\ref{thm:lowerbound} regarding the optimality of the lower bound.

Let $\cF$ be a function class, $\cP_\cF = \{\cP_f\}_{f \in \cF}$ a class of optimization problems (with some common form), and $\cE(\cdot;\cP_f)$ a suboptimality measure for the problem $\cP_f$.
Define the \emph{worst-case complexity} of an algorithm $\cA$ for $\cP_{\cF}$ at the $k$\nobreakdash-th iteration given the initial condition $\|\bz^0 - \bz^\star\| \le D$, as
\[
    \cC\left(\cA; \cP_{\cF}, D, k \right) := \sup_{\substack{\bz^0 \in B(\bz^\star; D)\\ f \in \cF}} \cE\left( \bz^k; \cP_f \right),
\]
where $\bz^j = \cA(\bz^0, \dots, \bz^{j-1}; f)$ for $j=1,\dots,k$ and $B(\bz;D)$ denotes the closed ball of radius $D$ centered at $\bz$.
The \emph{optimal complexity lower bound} with respect to an algorithm class $\mathfrak{A}$ is
\begin{align*}
    \cC \left(\mathfrak{A}; \cP_{\cF}, D, k \right) :&= \inf_{\cA \in \mathfrak{A}} \cC \left(\cA; \cP_{\cF}, D, k \right)\\
     &=\inf_{\cA \in \mathfrak{A}} \sup_{\substack{\bz^0 \in B(\bz^\star; D)\\ f \in \cF}} \cE\left( \bz^k; \cP_f \right).
\end{align*}
A \emph{complexity lower bound} is a lower bound on the optimal complexity lower bound.

Let $\cL_R(\reals^n \times \reals^m)$ be the class of $R$-smooth convex-concave functions on $\reals^n \times \reals^m$, $\cP_{\lagrange}$ the minimax problem \eqref{SPP_min_max}, and $\cE(\bz; \cP_\lagrange) = \|\nabla\lagrange (\bz)\|^2$.
With this notation, the results of Section \ref{section:Accelerated_algorithm} can be expressed as
\begin{align*}
    \cC \left(\textrm{EAG}; \cP_{\cL_R(\reals^n \times \reals^m)}, D, k \right) = \cO \left( \frac{R^2 D^2}{k^2} \right).
\end{align*}

Let $\cL_R^{\textrm{biaff}} (\reals^n \times \reals^m)$ be the class of $R$-smooth biaffine functions on $\reals^n \times \reals^m$.
Then the first statement of Theorem~\ref{thm:lowerbound}, the existence of $\lagrange$, can be expressed as
\begin{align}
    \label{complexity_lb_biaffine}
    \cC\left(\fA_{\textrm{sep}}; \cP_{\cL_R^{\textrm{biaff}}(\reals^n \times \reals^n)}, D, k \right) \ge \frac{R^2 D^2}{(2\lfloor k/2 \rfloor+1)^2}
\end{align}
for $n\ge k+2$.

As an aside, the argument of Corollary~\ref{cor:EAG-opt} can be expressed as:
for any $\mathcal{A}\in\fA_{\textrm{sep}}$, we have
\begin{align*}
    \cC\left(\mathcal{A}; \cP_{\cL_R(\reals^n \times \reals^n)}, D, k \right) &\ge\cC\left(\fA_{\textrm{sep}}; \cP_{\cL_R(\reals^n \times \reals^n)}, D, k \right)\\
     &\ge \cC\left(\fA_{\textrm{sep}}; \cP_{\cL_R^{\textrm{biaff}}(\reals^n \times \reals^n)}, D, k \right) \\
    &\ge \frac{R^2 D^2}{(2\lfloor k/2 \rfloor+1)^2}.
\end{align*}
The first inequality follows from $\mathcal{A}\in\fA_{\textrm{sep}}$,
the second from $\cL_R^{\textrm{biaff}} \subset \cL_R$,
and the third from Theorem~\ref{thm:lowerbound}.

\paragraph{Optimality of lower bound of Theorem~\ref{thm:lowerbound}.}
Using above notations, our goal is to prove that for $n\ge k+2$,
\begin{align}
    \label{complexity_lb_biaffine2}
    \cC\left(\fA_{\textrm{sep}}; \cP_{\cL_R^{\textrm{biaff}}(\reals^n \times \reals^n)}, D, k \right) = \frac{R^2 D^2}{(2\lfloor k/2 \rfloor+1)^2}.
\end{align}
We establish this claim with the chain of inequalities:
\begin{align}
    \frac{R^2 D^2}{(2\lfloor k/2 \rfloor+1)^2} & \le \cC\left(\fA_{\textrm{sep}}; \cP_{\cL_R^{\textrm{biaff}}(\reals^n \times \reals^n)}, D, k \right) \label{ineq:sep} \\
    &  \le \cC\left(\fA_{\textrm{sim}}; \cP_{\cL_R^{\textrm{biaff}}(\reals^n \times \reals^n)}, D, k \right) \label{ineq:sep_sim} \\
    &  \le \cC\left(\fA_{\textrm{lin}}; \cP^{2n,\textrm{skew}}_{R,D}, k \right) \label{ineq:sim_skew} \\
    &  \le \cC\left(\fA_{\textrm{lin}}; \cP^{2n}_{R,D}, k \right) \label{ineq:skew_full} \\
    &  \le \frac{R^2 D^2}{(2\lfloor k/2 \rfloor+1)^2}. \label{ineq:full_normaleq_chebyshev} 
\end{align}
Inequality \eqref{ineq:sep} is what we established in Section~\ref{ss:biaffine-worst-case}.
Inequality \eqref{ineq:sep_sim} follows from $\fA_{\textrm{sim}} \subset \fA_{\textrm{sep}}$ and the fact that the infimum over a larger class is smaller.
Roughly speaking, the quantities in lines \eqref{ineq:sim_skew} and \eqref{ineq:skew_full} are the complexity lower bounds for solving linear equations using only matrix-vector products, which were studied thoroughly in \citep{nemirovsky1991optimality, nemirovsky1992information}.
We will show inequalities \eqref{ineq:sim_skew}, \eqref{ineq:skew_full}, and \eqref{ineq:full_normaleq_chebyshev} by establishing the connection of Nemirovsky's work with our setup of biaffine saddle problems.
Once this is done, equality holds throughout and \eqref{complexity_lb_biaffine2} is proved.

We first provide the definitions.
Let $\cP^{2n}_{R,D}$ be the collection of linear equations with $2n \times 2n$ matrices $\bB$ satisfying $\|\bB\| \le R$ and $\bv = \bB\bz^\star$ for some $\bz^\star \in B(0;D)$.
Let $\cP^{2n,\textrm{skew}}_{R,D}\subset \cP^{2n}_{R,D}$ be the subclass of equations with skew-symmetric $\bB$.
Let $\fA_{\textrm{lin}}$ be the class of iterative algorithms solving linear equations $\bB\bz = \bv$ using only matrix multiplication by $\bB$ and $\bB^\intercal$ in the sense that
\begin{align}
    \label{eqn:nemirovsky_algorithm}
    \bz^k \in \spann \{\bv^0, \dots, \bv^{k} \},
\end{align}
where $\bv^0 = 0$, $\bv^1 = \bv$, and for $k \ge 2$,
\begin{align*}
    \bv^k = \bB\bv^{j} \text{ or } \bB^\intercal \bv^{j} \text{ for some } j=0,\dots,k-1.
\end{align*}
The optimal complexity lower bound for a class of linear equation instances is defined as
\begin{align*}
    \cC\left(\fA_{\textrm{lin}}; \cP^{2n}_{R,D}, k \right) := \inf_{\cA \in \fA_{\textrm{lin}}} \sup_{\substack{\|\bB\|\le R\\\bv=\bB\bz^\star, \|\bz^\star\|\le D}} \left\| \bB\bz^k - \bv \right\|^2.
\end{align*}
Define $ \cC\left(\fA_{\textrm{lin}}; \cP^{2n,\textrm{skew}}_{R,D}, k \right)$ analogously.

Now we relate the optimal complexity lower bounds for biaffine minimax problems to those for linear equations.
For $\lagrange(\bx,\by) = \bb^\intercal \bx + \bx^\intercal \bA\by - \bc^\intercal \by$, 
we have
\begin{align*}
    \sop_\lagrange(\bx,\by) =
    \begin{bmatrix}
    \bO & \bA \\ -\bA^\intercal & \bO
    \end{bmatrix}
    \begin{bmatrix}
    \bx \\ \by
    \end{bmatrix}
    + \begin{bmatrix}
    \bb \\ \bc
    \end{bmatrix}.
\end{align*}
Therefore, the minimax problem $\cP_{\lagrange}$ for $\lagrange \in \cL_R^{\textrm{biaff}}(\reals^n \times \reals^n)$ is equivalent to solving the linear equation $\bB\bz = \bv$ with 
$\bB = \begin{bmatrix}
    \bO & -\bA \\ \bA^\intercal & \bO
\end{bmatrix}$ and $\bv = \begin{bmatrix}
\bb \\ \bc
\end{bmatrix} \in \reals^{2n}$, which belongs to $\cP^{2n,\textrm{skew}}_{R,D}$ with $D = \|z^\star\|$.

For both algorithm classes $\fA_{\textrm{sim}}$ and $\fA_{\textrm{lin}}$, we may assume without loss of generality that $\bz^0 = 0$ through translation.
Then, the span condition \eqref{eqn:alg_span_sim} for $\fA_{\textrm{sim}}$ becomes
\begin{align}
    \label{eqn:Krylov_biaffine}
    \bz^k \in \mathcal K_{k-1} (\bB;\bv).
\end{align}
Note that \eqref{eqn:nemirovsky_algorithm} reduces to \eqref{eqn:Krylov_biaffine} as $\bB$ is skew-symmetric, so $\fA_{\textrm{sim}}$ and $\fA_{\textrm{lin}}$ are effectively the same class of algorithms under 
the identification $\cP_{\cL_R^{\textrm{biaff}}(\reals^n\times\reals^n)} \subset \cP_{R,D}^{2n,\textrm{skew}}$.

Since the supremum over a larger class of problems is larger, inequality \eqref{ineq:sim_skew} holds.
Similarly, 
inequality \eqref{ineq:skew_full} follows from $\cP^{2n,\textrm{skew}}_{R,D}\subset \cP^{2n}_{R,D}$.

\begin{figure*}[ht]
\vspace{-4mm}
\centering
\begin{tabular}{cc}
\hspace{-7mm}
\subfigure[Two-dimensional example $\lagrange_{\delta,\epsilon}$ of \eqref{eqn:L1}]{
      \includegraphics[width=0.48\textwidth]{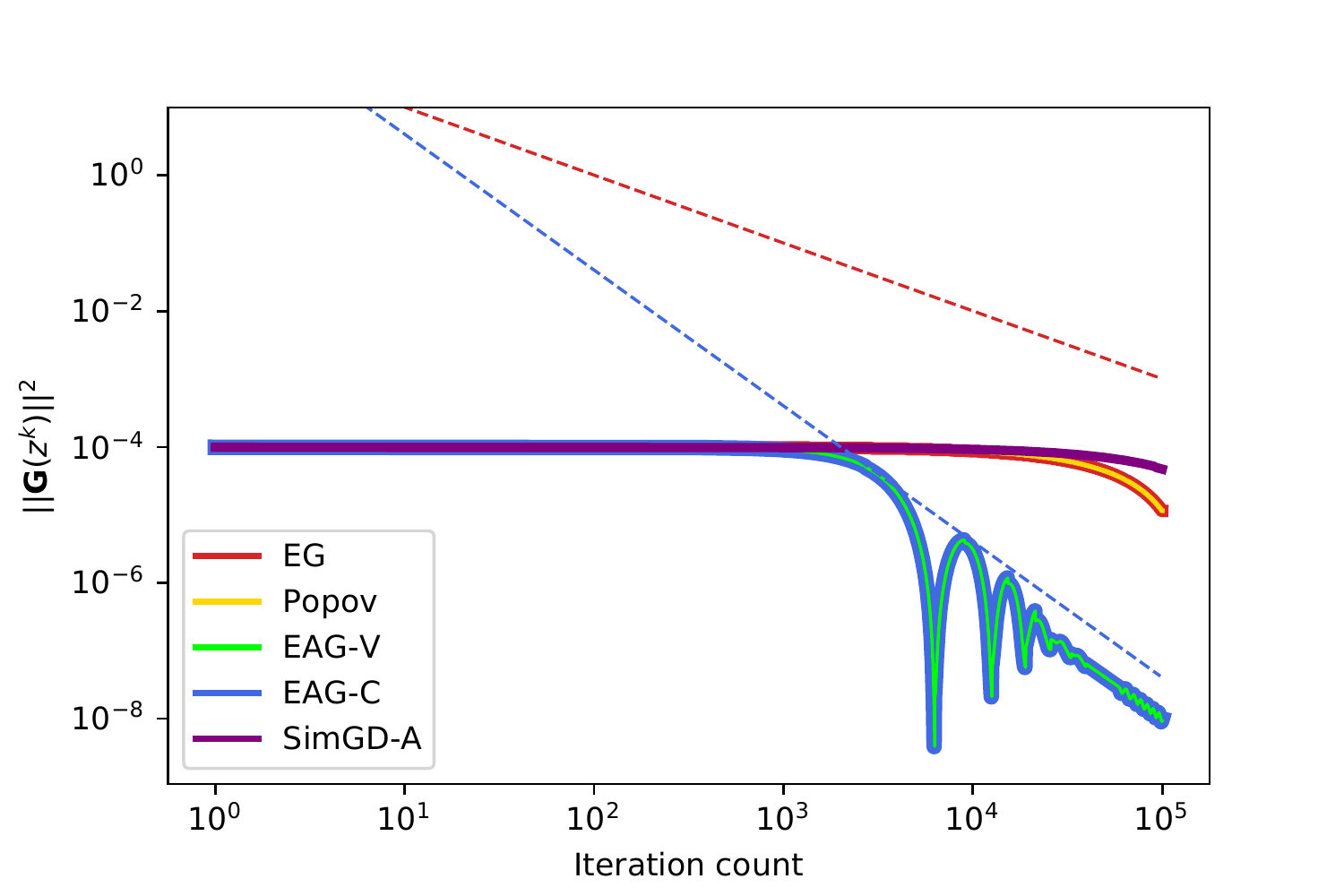}\label{fig:two_dim_loglog}}
&
\hspace{5mm}
\subfigure[Lagrangian of linearly constrained QP of \eqref{eqn:L2}]{
      \includegraphics[width=0.48\textwidth]{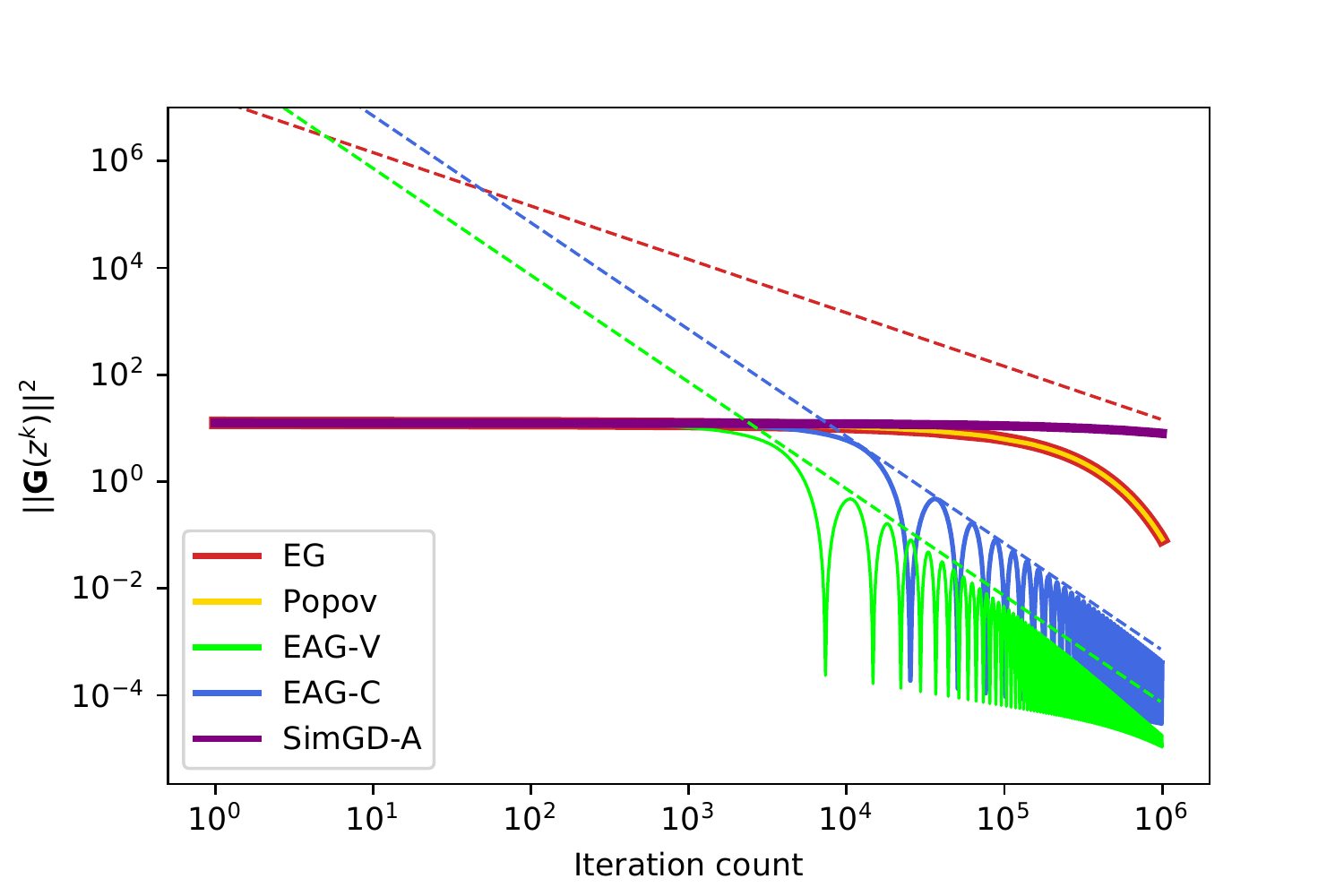}\label{fig:ouyang}}
\end{tabular}
 \vspace{-0.5em}
\caption[width=\textwidth]{Plots of $\|\sop (\bz^k)\|^2$ versus iteration count.
Dashed lines indicate corresponding theoretical upper bounds.
}
\label{fig:loglogs}
\end{figure*}

Finally, \eqref{ineq:full_normaleq_chebyshev} follows from the following lemma, using arguments based on Chebyshev-type matrix polynomials from \citet{nemirovsky1992information}.
Its proof is deferred to the appendix.
\begin{lemma}
\label{lemma:normaleq_chebyshev_algorithm}
Let $R>0$ and $k \ge 0$.
Then there exists $\cA \in \fA_{\textnormal{lin}}$ such that for any $m \ge 1$, $\bB \in \reals^{m\times m}$, and $\bv = \bB\bz^\star$ satisfying $\|\bB\|\le R$ and $\|\bz^\star\| \le D$, the $\bz^k$-iterate produced by $\cA$ satisfies
\begin{align*}
    \left\|\bB\bz^k - \bv\right\|^2 \le \frac{R^2 D^2}{(2\lfloor k/2 \rfloor+1)^2}.
\end{align*}
\end{lemma}

\subsection{Broader algorithm classes via resisting oracles}
In \eqref{eqn:alg_span_sim} and \eqref{eqn:alg-span}, we assumed the subgradient queries are made within the span of the gradients at the previous iterates.
This requirement (the \emph{linear span assumption}) can be removed, i.e., a similar analysis can be done on general deterministic black-box gradient-based algorithms (formally defined in the appendix, Section \ref{section:remove_span_condition}), using the resisting oracle technique \cite{nemirovsky1983problem} at the cost of slightly enlarging the required problem dimension.
We informally state the generalized result below and provide details in the appendix.

\begin{theorem}[Informal]
\label{thm:lowerbound_without_span_condition}
Let $n \ge 3k+2$. For any gradient-based deterministic algorithm, there exists an $R$-smooth biaffine function $\lagrange$ on $\reals^n \times \reals^n$ such that \eqref{eqn:thm_biaffine_lowerbound} holds.
\end{theorem}

Although we do not formally pursue this, the requirement that the algorithm is not randomized can also be removed using the techniques of \citet{woodworth2016tight}, which exploit near-orthogonality of random vectors in high dimensions.

\subsection{Discussion}
We established that one cannot improve the lower bound of Theorem \ref{thm:lowerbound} using biaffine functions, arguably the simplest family of convex-concave functions.
Furthermore, this optimality statement holds for both algorithm classes $\fA_{\textrm{sep}}$ and $\fA_{\textrm{sim}}$ as established through the chain of inequalities in Section~\ref{ss:optimal-complexity}.
However, as demonstrated by \citet{drori2017exact}, who introduced a non-quadratic lower bound for smooth convex minimization that improves upon the classical quadratic lower bounds of \citet{nemirovsky1992information} and \citet{nesterov2013introductory}, a non-biaffine construction may improve the constant.
In our setup, there is a factor-near-$100$ difference between the upper and lower bounds. (Note that each EAG iteration requires $2$ evaluations of the saddle subdifferential oracle.)
We suspect that both the algorithm and the lower bound can be improved upon, but we leave this to future work.

\citet{golowich2020last} establishes that for the class of 1-SCLI algorithms (S is for \emph{stationary}), a subclass of $\fA_{\textrm{sim}}$ for biaffine objectives, one cannot achieve a rate faster than $\|\nabla \lagrange (\bz^k) \|^2 \le \mathcal{O} (1/k)$.
This lower bound applies to EG but not EAG;
EAG is not 1-SCLI, as its anchoring coefficients $\frac{1}{k+2}$ vary over iterations, and its convergence rate breaks the 1-SCLI lower bound.
On the other hand, we can view EAG as a non-stationary CLI algorithm \citep[Definition 2]{arjevani2016iteration}.
We further discuss these connections in the appendix, Section \ref{section:CLI}.

\begin{figure*}[ht]
\vspace{-3mm}
\centering
\begin{tabular}{cc}
\hspace{-7mm}
\subfigure[Discrete trajectories with $\lagrange_{\delta,\epsilon}$]{
      \includegraphics[width=0.5\textwidth]{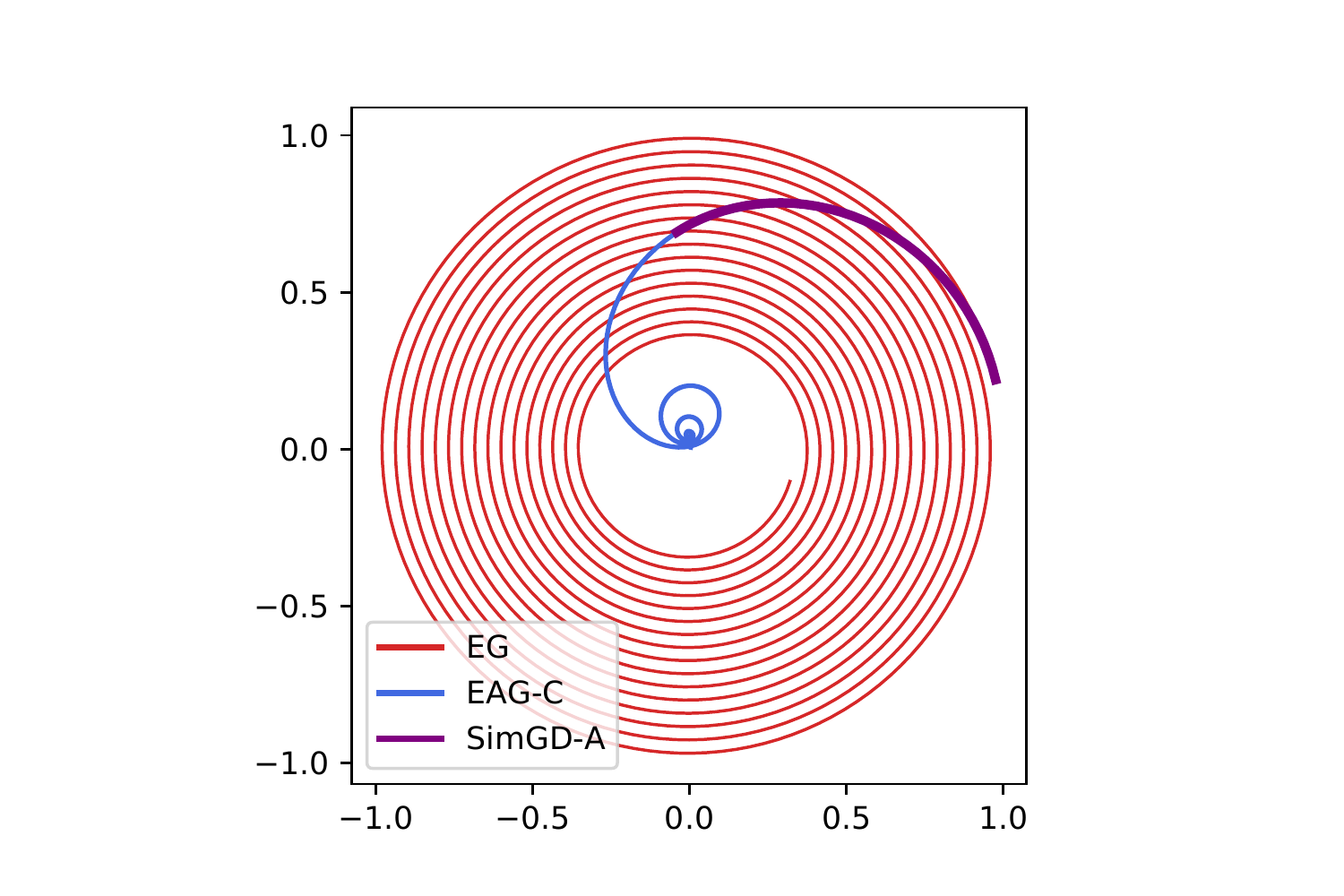}\label{fig:two_dim_scatter}}
&
\subfigure[Moreau--Yosida regularized flow with $\lambda=0.01$ and the anchored flow with $\lagrange(x,y) = xy$]{
      \includegraphics[width=0.5\textwidth]{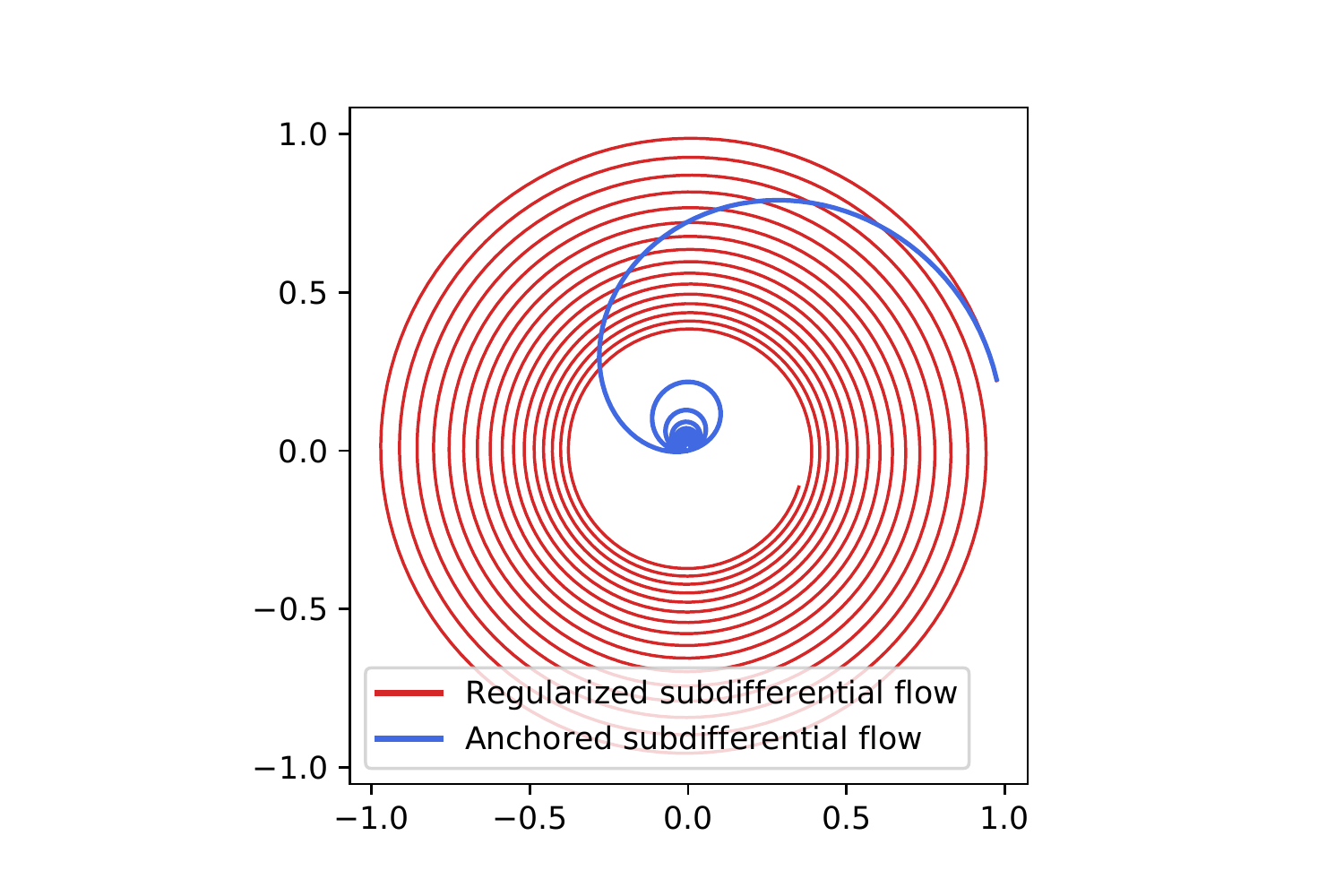}\label{fig:two_dim_ode}}
\end{tabular}
\vspace{-0.5em}
\caption{
Comparison of the discrete trajectories and their corresponding continuous-time flow.
Trajectories from EAG-C and SimGD-A virtually coincide and resemble the anchored flow.
However, SimGD-A progresses slower due to its diminishing step-sizes.
}
\vspace{-0.5em}
\label{fig:trajectories}
\end{figure*}

\section{Experiments}
\label{section:experiments}
We now present experiments illustrating the accelerated rate of EAG.
We compare EAG-C and EAG-V against the prior algorithms with convergence guarantees: EG, Popov's algorithm (or optimistic descent)
and simultaneous gradient descent with anchoring (SimGD-A).
The precise forms of the algorithms are restated in the appendix.

Figure \ref{fig:two_dim_loglog} presents experiments on our first example, constructed as follows.
For $\epsilon>0$, define
\begin{align*}
    f_\epsilon (u) = \begin{cases}
        \epsilon |u| - \frac{1}{2} \epsilon^2   &  \text{if } |u| \ge \epsilon, \\
        \frac{1}{2} u^2                     &  \text{if } |u| < \epsilon.
    \end{cases}
\end{align*}
Next, for $0<\epsilon \ll \delta \ll 1$, define
\begin{align}
    \label{eqn:L1}
    \lagrange_{\delta,\epsilon} (x,y) = (1-\delta) f_\epsilon(x) + \delta xy - (1-\delta) f_\epsilon(y),
\end{align}
where $x,y\in\reals$.
Since $f_\epsilon$ is a $1$-smooth convex function, $\lagrange_{\delta,\epsilon}$ has 
smoothness parameter $1$, which is almost tight due to the quadratic behavior of $\lagrange_{\delta,\epsilon}$ within the region $|x|,|y|\le \epsilon$.
This construction was inspired by \citet{drori2014performance}, who presented $f_\epsilon$ as the worst-case instance for gradient descent.
We choose the step-size $\alpha=0.1$ as this value is comfortably within the theoretical range of convergent parameters for EG, EAG-C, and Popov.
For EAG-V, we set $\alpha_0 = 0.1$.
We use $N=10^5$, $\delta = 10^{-2}$, and $\epsilon = 5 \times 10^{-5}$, and the initial point $\bz^0$ has norm $1$.

Figure \ref{fig:ouyang} presents experiments on our second example
\begin{align}
    \label{eqn:L2}
    \lagrange(\bx,\by) =  \frac{1}{2} \bx^\intercal \bH \bx - \bh^\intercal \bx - \langle \bA\bx-\bb, \by \rangle,    
\end{align}
where $\bx,\by \in \reals^n$, $\bA \in \reals^{n\times n}$, $\bb \in \reals^n$, $\bH \in \reals^{n\times n}$ is positive semidefinite, and $\bh \in \reals^n$.
Note that this is the Lagrangian of a linearly constrained quadratic minimization problem.
We adopted this saddle function from \citet{ouyang2019lower}, where the authors constructed $\bH, \bh, \bA$ and $\bb$ to provide a lower bound on duality gap. 
The exact forms of $\bH$, $\bh$, $\bA$, and $\bb$ are restated in the appendix.
We use $n=200$, $N=10^6$, $\alpha=0.5$ for EG and Popov, $\alpha=0.1265$ for EAG-C and $\alpha_0 = 0.618$ for EAG-V.
Finally, we use the initial point $\bz^0 = 0$.

\paragraph{ODE Interpretation}

Figure \ref{fig:two_dim_scatter} illustrates the algorithms applied to \eqref{eqn:L1}.
For $|x|,|y| \gg \epsilon$, 
\[
\sop_{\lagrange_{\delta,\epsilon}}(x,y) = \begin{bmatrix}
(1-\delta) \epsilon + \delta y \\ (1-\delta) \epsilon - \delta x
\end{bmatrix} \approx \delta \begin{bmatrix} y \\ -x \end{bmatrix},
\]
so the algorithms roughly behave as if the objective is the bilinear function $\delta xy$.
When $\delta$ is sufficiently small, trajectories of the algorithms closely resemble the corresponding continuous-time flows with $\lagrange(x,y) = xy$.

\citet{csetnek2019shadow} demonstrated that Popov's algorithm can be viewed as discretization of the Moreau--Yosida regularized flow $\dot{\bz}(t) = -\frac{\sop - (\textrm{Id}+\lambda \sop)^{-1}}{\lambda} \left(\bz(t)\right)$ for some $\lambda > 0$, and a similar analysis can be performed with EG.
This connection explains why EG's trajectory in Figure~\ref{fig:two_dim_scatter} and the regularized flow depicted in Figure~\ref{fig:two_dim_ode} are similar.

On the other hand, EAG and SimGD-A can be viewed as a discretization of the anchored flow ODE
\[
    \dot{\bz}(t) = -\sop (\bz(t)) + \frac{1}{t} (\bz^0 - \bz(t)).
\]
The anchored flow depicted in Figure~\ref{fig:two_dim_ode} approaches the solution much more quickly due to the anchoring term dampening the cycling behavior.
The trajectories of EAG and SimGD-A iterates in Figure~\ref{fig:two_dim_scatter} are very similar to the anchored flow.
However, SimGD-A requires diminishing step-sizes $\frac{1-p}{(k+1)^p}$ (both theoretically and experimentally) and therefore progresses much slower.

\section{Conclusion}
This work presents the extra anchored gradient (EAG) algorithms, which exhibit accelerated $\mathcal{O}(1/k^2)$ rates on the squared gradient magnitude for smooth convex-concave minimax problems. The acceleration combines the extragradient and anchoring mechanisms, which separately achieve $\mathcal{O}(1/k)$ or slower rates. We complement the $\mathcal{O}(1/k^2)$ rate with a matching $\Omega(1/k^2)$ complexity lower bound, thereby establishing optimality of EAG.

At a superficial level, the acceleration mechanism of EAG seems to be distinct from that of Nesterov; anchoring dampens oscillations, but momentum provides the opposite effect of dampening. However, are the two accelerations phenomena entirely unrelated? Finding a common structure, a connection, between the two acceleration phenomena would be an interesting direction of future work.

\section*{Acknowledgements}
TY and EKR were supported by the National Research Foundation of Korea (NRF) Grant funded by the Korean Government (MSIP) [No. 2020R1F1A1A01072877], the National Research Foundation of Korea (NRF) Grant funded by the Korean Government (MSIP) [No. 2017R1A5A1015626], by the New Faculty Startup Fund from Seoul National University, and by the AI Institute of Seoul National University (AIIS) through its AI Frontier Research Grant (No. 0670-20200015) in 2020.
We thank Jaewook Suh and Jongmin Lee for reviewing the manuscript and providing valuable feedback. 
We thank Jelena Diakonikolas for the discussion on the prior work on parameter-free near-optimal methods for the smooth minimax setup.
Finally, we thank the anonymous referees for bringing to our attention the recent complexity lower bound on the class of $1$-SCLI algorithms by \citet{golowich2020last}.

\newpage

\bibliography{Accelerated_Saddle}
\bibliographystyle{icml2021}

\onecolumn

\newpage

\appendix

\section{Algorithm specifications}
\label{section:algorithm_specifications}

For the sake of clarity, we precisely specify all the algorithms discussed in this work.

\emph{Simultaneous gradient descent} for smooth minimax optimization is defined as
\begin{align*}
    & \bx^{k+1} = \bx^k - \alpha \nabla_\bx \lagrange(\bx^k,\by^k) \\ 
    & \by^{k+1} = \by^k + \alpha \nabla_\by \lagrange(\bx^k,\by^k).
\end{align*}
The notation becomes more concise with the joint variable notation $\bz^k = (\bx^k,\by^k)$ and the saddle operator \eqref{eqn:saddle_subdifferential}, where the sign change in $\by$-gradient is already included:
\begin{align*}
    \bz^{k+1} = \bz^k - \alpha \sop (\bz^k).
\end{align*}
\emph{Alternating gradient descent-ascent} is defined as
\begin{align*}
    & \bx^{k+1} = \bx^k - \alpha \nabla_\bx \lagrange(\bx^k,\by^k)\\
    & \by^{k+1} = \by^k + \alpha \nabla_\by \lagrange(\bx^{k+1},\by^k).
\end{align*}
Note that we update the $\bx$ variable first and then use it to update the $\by$-iterate.

The \emph{extragradient (EG) algorithm} is defined as
\begin{align*}
    \bz^{k+1/2} & = \bz^k - \alpha \sop (\bz^k), \\
    \bz^{k+1} & = \bz^k - \alpha \sop (\bz^{k+1/2}).
\end{align*}

\emph{Popov's algorithm}, or \emph{optimistic descent}, is defined as
\begin{align}
    \nonumber
    \bz^{k+1} = \bz^k - \alpha \sop (\bz^k) - \alpha \left( \sop(\bz^k) - \sop(\bz^{k-1}) \right).
\end{align}
\emph{Simultaneous gradient descent with anchoring (SimGD-A)} \cite{ryu2019ode} is defined as
\begin{align*}
    \bz^{k+1} = \bz^k - \frac{1-p}{(k+1)^p} \sop(\bz^k) + \frac{(1-p)\gamma}{k+1} (\bz^0 - \bz^k),
\end{align*}
where $p \in (1/2,1)$ and $\gamma > 0$.
It has been proved in \citet{ryu2019ode} that SimGD-A converges at $\mathcal{O}(1/k^{2-2p})$ rate.
In this paper, we always used $\gamma = 1$ and $p = \frac{1}{2} + 10^{-2}$.

\section{Omitted proofs of Section~\ref{section:Accelerated_algorithm}}

The following identities follow directly from the definition of EAG iterates:
\begin{gather}
\bz^k - \bz^{k+1} = \beta_k (\bz^k - \bz^0) + \alpha_k \, \sop (\bz^{k+1/2}) \label{eqn:difference_zk_zkp1}\\    
\bz^{k+1/2} - \bz^{k+1} = \alpha_k \left(\sop(\bz^{k+1/2}) - \sop(\bz^k)\right) \label{eqn:difference_zk_zkhalf}\\
\bz^0 - \bz^{k+1} = (1-\beta_k) (\bz^0 - \bz^k) + \alpha_k \, \sop(\bz^{k+1/2}). \label{eqn:difference_z0_zkp1}
\end{gather}

\subsection{Proof of Lemma~\ref{lemma:lyap}}
Recall that $\sop$ is a monotone operator, so that
\begin{align*}
    0 &\leq \left\langle \bz^k - \bz^{k+1}, \sop(\bz^k) - \sop(\bz^{k+1}) \right\rangle.
\end{align*}
Therefore,
\begin{align}
    & V_k - V_{k+1} \nonumber \\
    & \geq V_k - V_{k+1} - \frac{B_k}{\beta_k} \left\langle \bz^k - \bz^{k+1}, \sop(\bz^k) - \sop(\bz^{k+1}) \right\rangle \nonumber \\
    &= A_k \left\|\sop(\bz^k)\right\|^2 + B_k \left\langle \sop (\bz^k), \bz^k - \bz^0 \right\rangle \nonumber \\
    & \quad - A_{k+1} \left\|\sop(\bz^{k+1})\right\|^2 - B_{k+1} \left\langle \sop (\bz^{k+1}), \bz^{k+1} - \bz^0 \right\rangle - \frac{B_k}{\beta_k} \left\langle \bz^k - \bz^{k+1}, \sop(\bz^k) - \sop(\bz^{k+1}) \right\rangle \nonumber \\
    & \labelrel={ineq:lemma_lyap_first} A_k \left\|\sop(\bz^k)\right\|^2 + B_k \left\langle \sop (\bz^k), \bz^k - \bz^0 \right\rangle \nonumber \\
    & \quad - A_{k+1} \left\|\sop(\bz^{k+1})\right\|^2 + B_{k+1} \left\langle \sop (\bz^{k+1}), (1-\beta_k) (\bz^0 - \bz^k) + \alpha_k \, \sop(\bz^{k+1/2}) \right\rangle \nonumber \\
    & \quad - B_k \left\langle \bz^k - \bz^0, \sop(\bz^k) - \sop(\bz^{k+1}) \right\rangle - \frac{\alpha_k B_k}{\beta_k} \left\langle \sop(\bz^{k+1/2}), \sop(\bz^k) - \sop(\bz^{k+1}) \right\rangle \nonumber \\
    \label{eqn:lemma-ineq-mono}
    & \begin{aligned}
        & \labelrel={ineq:lemma_lyap_second} A_k \left\|\sop(\bz^k)\right\|^2 - A_{k+1} \left\|\sop(\bz^{k+1})\right\|^2 + \alpha_k B_{k+1} \left\langle \sop(\bz^{k+1}), \sop(\bz^{k+1/2}) \right\rangle \\
        & \quad - \frac{\alpha_k B_k}{\beta_k} \left\langle \sop(\bz^{k+1/2}), \sop(\bz^k) - \sop(\bz^{k+1}) \right\rangle,
    \end{aligned}
\end{align}
where (\ref{ineq:lemma_lyap_first}) follows from \eqref{eqn:difference_zk_zkp1} and \eqref{eqn:difference_z0_zkp1}, and \eqref{ineq:lemma_lyap_second} results from cancellation and collection of terms using \eqref{eqn:lemma-recur-bk}.
Next, we have
\begin{align}
    \label{eqn:lemma-Lip}
    \begin{aligned}
        0 & \leq R^2 \big\|\bz^{k+1/2} - \bz^{k+1}\big\|^2 - \big\|\sop(\bz^{k+1/2}) - \sop(\bz^{k+1})\big\|^2 \\
        & = \alpha_k^2 R^2 \, \big\|\sop(\bz^k) - \sop(\bz^{k+1/2})\big\|^2 - \big\|\sop(\bz^{k+1/2}) - \sop(\bz^{k+1})\big\|^2
    \end{aligned}
\end{align}
from $R$-Lipschitzness of $\sop$ and \eqref{eqn:difference_zk_zkhalf}.
Now multiplying the factor $\frac{A_k}{\alpha_k^2 R^2}$ to \eqref{eqn:lemma-Lip} and subtracting from \eqref{eqn:lemma-ineq-mono} gives
\begin{align}
    \label{eqn:lemma-ineq-Lip}
    & V_k - V_{k+1} \nonumber \\
    & \geq A_k \left\|\sop(\bz^k)\right\|^2 - A_{k+1} \left\|\sop(\bz^{k+1})\right\|^2 + \alpha_k B_{k+1} \big\langle \sop(\bz^{k+1}), \sop(\bz^{k+1/2}) \big\rangle \nonumber \\
    & \quad - \frac{\alpha_k B_k}{\beta_k} \left\langle \sop(\bz^{k+1/2}), \sop(\bz^k) - \sop(\bz^{k+1}) \right\rangle \nonumber \\
    & \quad - A_k \, \left\|\sop(\bz^k) - \sop(\bz^{k+1/2})\right\|^2 + \frac{A_k}{\alpha_k^2 R^2} \left\|\sop(\bz^{k+1/2}) - \sop(\bz^{k+1})\right\|^2 \nonumber \\
    & \begin{aligned}
        & = \frac{A_k (1 - \alpha_k^2 R^2)}{\alpha_k^2 R^2} \left\|\sop(\bz^{k+1/2})\right\|^2 + \left( \frac{A_k}{\alpha_k^2 R^2} - A_{k+1} \right) \left\|\sop(\bz^{k+1})\right\|^2 \\
        & \quad + \left( 2A_k - \frac{\alpha_k B_k}{\beta_k} \right) \left\langle \sop(\bz^k), \sop(\bz^{k+1/2}) \right\rangle \\
        & \quad + \left( \alpha_k B_{k+1} + \frac{\alpha_k B_k}{\beta_k} - \frac{2 A_k}{\alpha_k^2 R^2} \right)\left\langle \sop(\bz^{k+1/2}), \sop(\bz^{k+1}) \right\rangle.
    \end{aligned}
\end{align}
Observe that the $\left\langle \sop(\bz^k), \sop(\bz^{k+1/2}) \right\rangle$ term vanishes because of \eqref{eqn:lemma-recur-ak}, and that
\begin{align*}
    \alpha_k B_{k+1} + \frac{\alpha_k B_k}{\beta_k} = \alpha_k \left( \frac{B_k}{1-\beta_k} + \frac{B_k}{\beta_k} \right) = \frac{\alpha_k B_k}{\beta_k (1-\beta_k)} = \frac{2A_k}{1-\beta_k}.
\end{align*}
Furthermore, by \eqref{eqn:lemma-recur-alphak}, we have
\begin{align*}
    A_{k+1} = \alpha_{k+1}  \frac{B_{k+1}}{2\beta_{k+1}}  =  \frac{\alpha_k \beta_{k+1} (1-\alpha_k^2 R^2-\beta_k^2)}{(1-\alpha_k^2 R^2) \beta_k (1-\beta_k)}  \frac{B_k}{2\beta_{k+1}(1-\beta_k)} = \frac{A_k (1-\alpha_k^2 R^2-\beta_k^2)}{(1-\alpha_k^2 R^2)(1-\beta_k)^2}.
\end{align*}
Plugging these identities into \eqref{eqn:lemma-ineq-Lip} and simplifying, we get
\begin{align*}
    & V_k - V_{k+1} \nonumber \\
    & \geq \frac{A_k (1 - \alpha_k^2 R^2)}{\alpha_k^2 R^2} \left\|\sop(\bz^{k+1/2})\right\|^2 + \frac{A_k (1 - \alpha_k^2 R^2 - \beta_k)^2}{\alpha_k^2 R^2 (1-\alpha_k^2 R^2) (1-\beta_k)^2} \left\|\sop(\bz^{k+1})\right\|^2\\
    & \quad - \frac{2A_k (1-\alpha_k^2 R^2 - \beta_k)}{\alpha_k^2 R^2 (1-\beta_k)} \left\langle \sop(\bz^{k+1/2}), \sop(\bz^{k+1}) \right\rangle \\
    & \geq 0,
\end{align*}
where the last inequality is an application of Young's inequality.

\subsection{Proof of Lemma~\ref{lemma:alpha}}
We may assume $R=1$ without loss of generality because we can recover the general case by replacing $\alpha_k$ with $\alpha_k R$.
Rewrite \eqref{eqn:recurrence-alpha} as
\begin{align}
    \label{eqn:recurrence-simple}
    \alpha_k - \alpha_{k+1} = \frac{\alpha_k^3}{(k+1)(k+3)(1-\alpha_k^2)}.
\end{align}
Suppose that we have already established $0 < \alpha_N < \rho$ for some $N\geq 0$ and $\rho \in (0,1)$, where $\rho$ satisfies
\begin{align}
    \label{eqn:gamma_le_1}
    \gamma := \frac{1}{2} \left( \frac{1}{N+1} + \frac{1}{N+2} \right) \frac{\rho^2}{1-\rho^2} < 1.
\end{align}
Note that \eqref{eqn:gamma_le_1} holds true for all $N \ge 0$ if $\rho < \frac{3}{4}$.
Now we will show that given \eqref{eqn:gamma_le_1},
\begin{align*}
    \alpha_N > \alpha_{N+1} > \cdots > \alpha_{N+k} > (1-\gamma) \alpha_N \quad \text{for all }k \ge 0,
\end{align*}
so that $\alpha_{k} \downarrow \alpha$ for some $\alpha \geq (1-\gamma) \alpha_N$.
It suffices to prove that $(1-\gamma)\alpha_N < \alpha_{N+k} < \rho$ for all $k\geq 0$, because it is clear from \eqref{eqn:recurrence-simple} that $\{\alpha_{k}\}_{k \ge 0}$ is decreasing.

We use induction on $k$ to prove that $\alpha_{N+k} \in ((1-\gamma)\alpha_N, \rho)$. The case $k=0$ is trivial. Now suppose that $(1-\gamma)\alpha_N < \alpha_{N+j} < \rho$ holds true for all $j = 0,\dots,k$.
Then by \eqref{eqn:recurrence-simple}, for each $0\leq j\leq k$ we have
\begin{align*}
    0 < \alpha_{N+j} - \alpha_{N+j+1} &= \frac{1}{(N+j+1)(N+j+3)} \frac{\alpha_{N+j}^3}{1-\alpha_{N+j}^2} \\
    & < \frac{1}{(N+j+1)(N+j+3)} \frac{\rho^2 \alpha_N}{1-\rho^2}.
\end{align*}
Summing up the inequalities for $j=0,\dots,k$, we obtain
\begin{align*}
    0 < \alpha_N - \alpha_{N+k+1} &< \sum_{j=0}^k \frac{1}{(N+j+1)(N+j+3)} \frac{\rho^2 \alpha_N}{1-\rho^2}\\
    &< \frac{\rho^2 \alpha_N}{1-\rho^2} \sum_{j=0}^\infty \frac{1}{(N+j+1)(N+j+3)}\\
    & = \frac{\rho^2 \alpha_N}{1-\rho^2} \frac{1}{2} \left( \frac{1}{N+1} + \frac{1}{N+2} \right) = \gamma \alpha_N,
\end{align*}
which gives $(1-\gamma) \alpha_N < \alpha_{N+k+1} < \alpha_N < \rho$, completing the induction.

In particular, when $\alpha_0 = 0.618$, direct calculation gives $0.437 > \alpha_N > 0.4366$ when $N=1000$.
With $\rho = 0.437$ and $N=1000$, we have $\gamma = \frac{1}{2} \left( \frac{1}{N+1} + \frac{1}{N+2} \right) \frac{\rho^2}{1-\rho^2} < 2.5 \times 10^{-4}$, which gives $\alpha \ge (1-\gamma)\alpha_N \approx 0.4365$.

\subsection{Proof of Theorem \ref{thm:EAG-C}}

As in the proof of Theorem~\ref{thm:EAG-V}, assume without loss of generality that $R=1$.
The strategy of the proof is basically the same as in Theorem \ref{thm:EAG-V}; we construct a nonincreasing Lyapunov function by combining the same set of inequalities, but with different (more intricate) coefficients.
For $k\geq 0$, let
\begin{align*}
    V_k = A_k \left\|\sop (\bz^k)\right\|^2 + B_k \left\langle \sop(\bz^k), \bz^k - \bz^0 \right\rangle.
\end{align*}
As in Lemma~\ref{lemma:lyap}, we will use $B_k = \frac{1}{1-\beta_k} = k+1$, and $a_k \geq 0$ will be specified later.
Because we have the fixed step-size $\alpha$, the identities \eqref{eqn:difference_zk_zkp1}, \eqref{eqn:difference_zk_zkhalf}, and \eqref{eqn:difference_z0_zkp1} become
\begin{gather*}
    \bz^{k+1/2} - \bz^{k+1} = \alpha \left( \sop(\bz^{k+1/2}) - \sop(\bz^k) \right)\\
    \bz^k - \bz^{k+1} = \frac{1}{k+2} (\bz^k - \bz^0) + \alpha \, \sop (\bz^{k+1/2})\\
    \bz^{k+1} - \bz^0 = \frac{k+1}{k+2} (\bz^k - \bz^0) - \alpha \, \sop(\bz^{k+1/2}).
\end{gather*}
Now, subtracting the same inequalities from monotonicity and Lipschitzness from $V_k - V_{k+1}$ as in Lemma~\ref{lemma:lyap}, each with coefficients $(k+1)(k+2)$ and $\tau_k \ge 0$ (to be specified later), we obtain
\begin{align*}
    & V_k - V_{k+1} \\
    & \geq V_k - V_{k+1} - (k+1)(k+2) \left\langle \bz^k - \bz^{k+1}, \sop(\bz^k) - \sop(\bz^{k+1}) \right\rangle \\
    & \quad - \tau_k \, \left( \left\|\bz^{k+1/2} - \bz^{k+1}\right\|^2 - \left\|\sop(\bz^{k+1/2}) - \sop(\bz^{k+1})\right\|^2 \right)\\
    &= (A_k - \alpha^2 \tau_k) \left\|\sop(\bz^k) \right\|^2 + \tau_k (1-\alpha^2) \left\|\sop(\bz^{k+1/2}) \right\|^2 + (\tau_k - A_{k+1}) \left\|\sop(\bz^{k+1}) \right\|^2\\
    & \quad + \left( 2\alpha^2 \tau_k - \alpha(k+1)(k+2) \right)\, \left\langle \sop(\bz^k), \sop(\bz^{k+1/2}) \right\rangle + \left(\alpha(k+2)^2 - 2 \tau_k \right) \, \left\langle \sop(\bz^{k+1/2}), \sop(\bz^{k+1}) \right\rangle\\
    &= \Tr \left( \mathbf{M}_k \mathbf{S}_k \mathbf{M}_k^\intercal \right),
\end{align*}
where we define $\mathbf{M}_k := \begin{bmatrix}
\sop(\bz^k) & \sop(\bz^{k+1/2}) & \sop(\bz^{k+1})
\end{bmatrix}$
and
\begin{align}
    \label{eqn:Sk}
    \mathbf{S}_k := \begin{bmatrix}
    A_k - \alpha^2 \tau_k & \alpha^2 \tau_k - \frac{\alpha}{2}(k+1)(k+2) & 0 \\
    \alpha^2 \tau_k - \frac{\alpha}{2}(k+1)(k+2) & \tau_k (1-\alpha^2) & \frac{\alpha}{2}(k+2)^2 - \tau_k \\
    0 & \frac{\alpha}{2}(k+2)^2 - \tau_k & \tau_k - A_{k+1}
    \end{bmatrix}.
\end{align}
If $\mathbf S_k \succeq \bO$, then $\Tr \left( \mathbf{M}_k \mathbf{S}_k \mathbf{M}_k^\intercal \right) = \Tr \left( \mathbf{S}_k \mathbf{M}_k^\intercal \mathbf{M}_k \right) \ge 0$ because the positive semidefinite cone is self-dual with respect to the matrix inner product $\langle \bA, \bB\rangle = \Tr (\bA^\intercal \bB)$.
Because $B_k = k+1$ grows linearly, provided that the sequence $\{A_k\}$ grows quadratically, we can derive $\mathcal{O}(1/k^2)$ convergence by using similar line of arguments as in the proof of Theorem \ref{thm:EAG-V}.
This reduction of the proof into a search of appropriate parameters (i.e., $\tau_k$) that meet semidefiniteness constraints ($\mathbf{S}_k \succeq \bO$ in our case) while allowing for desired rate of growth in Lyapunov function coefficients ($A_k$ in our case) was inspired by works of \citet{taylor2017smooth} and \citet{taylor2019stochastic}.
In the following, we demonstrate that careful choices of $A_0$ and $\tau_k$ make $A_k$ asymptotically close to $\frac{\alpha(k+1)(k+2)}{2}$, so quadratic growth is guaranteed.
We begin with the following lemma, which will be used throughout the proof.

\begin{lemma}
\label{lemma:lkuk}
Let $k\in \mathbb N_{\geq 0}$ and $\alpha \in \left(0,\frac{1}{2}\right]$ be fixed, and define
\begin{align*}
    \ell_k := \frac{\alpha(k+2)(k+1+k\alpha)}{2(1+\alpha)}, \quad u_k := \frac{\alpha(k+2)(k+1-k\alpha)}{2(1-\alpha)}.
\end{align*}
Then,
\begin{align}
    u_k > \frac{\alpha(k+1)(k+2)}{2} & > \ell_k \label{eqn:interval}\\
    & \geq \frac{\alpha(k+1)(k+1+\alpha(k+2))}{2(1+\alpha)} \label{eqn:tau1_positivity} \\
    & \geq \frac{\alpha(k+1)^2 - \alpha^3 k(k+2)}{2(1-\alpha^2)} 
    \label{eqn:tau_comparison}\\
    & \geq \max \left\{ \frac{\alpha(k+1)(k+1-\alpha(k+2))}{2(1-\alpha)} , \frac{\alpha^2 (k+1)(k+2)}{1+\alpha} \right\} \label{tau2_positivity}\\
    & \geq \frac{\alpha^2 (k+1)(k+2) + \alpha^3 (k+2)^2}{2(1+\alpha)}. \label{tau1_upperbound}
\end{align}
\end{lemma}

We shall prove Lemma~\ref{lemma:lkuk} after the proof of the main theorem and for now, focus on why we need such results.
Observe that all the quantities within the lines \eqref{eqn:interval} through \eqref{eqn:tau_comparison} are asymptotically close to $\frac{\alpha k^2}{2}$.
We show that $A_k \in I_k := [\ell_k, u_k]$ for all $k \ge 0$, which implies the quadratic growth.
The quantities in Lemma~\ref{lemma:lkuk} are used for choosing the right $\tau_k$ and for showing the positive semidefiniteness of $\mathbf{S}_k$.

Subdivide the interval $I_k$ into two parts:
\begin{align*}
    I_k^- = \left[ \ell_k, \frac{\alpha(k+1)(k+2)}{2} \right], \quad I_k^+ = \left[ \frac{\alpha(k+1)(k+2)}{2}, u_k \right].
\end{align*}
We divide cases: $A_k \in I_k^-$ and $A_k \in I_k^+$.
However, the latter case is in fact not needed unless we wish to extend the proof for $\alpha$ beyond $\frac{0.1265}{R}$.
If that is not the case, we recommend the readers to refer to Case 1 only.
Nevertheless, we exhibit analysis of both cases because Case 2 might provide useful data for enlarging or even completely determining the range of convergent step-sizes for EAG-C.

\noindent\textbf{Case 1.}
Suppose that $A_k \in I_k^-$.
In this case, we choose 
\begin{align}
    \label{def:tau1}
    \tau_k = \frac{ (k+2)^2  \left( 2(1-\alpha)A_k - \alpha(k+1)(k+1-\alpha(k+2)) \right) }
    { 2 \left(\alpha(k+2)(k+1-k\alpha) - 2(1-\alpha) A_k \right) }.
\end{align}
The denominator and numerator of (\ref{def:tau1}) are both positive because $u_k > A_k > \frac{\alpha(k+1)(k+1-\alpha(k+2))}{2(1-\alpha)}$ (see (\ref{tau2_positivity})).
Thus, $\tau_k > 0$.
Next, define $A_{k+1}$ as
\begin{align}
    A_{k+1} &= \frac{\alpha(k+2)^2 \left( 4(1-\alpha) A_k - \alpha(k+1-\alpha(k+2))^2 \right)}
    {4(1-\alpha) \left( (1-\alpha)A_k + \alpha^2(k+1)(k+2) \right)} \nonumber \\
    &= \frac{\alpha (k+2)^2}{1-\alpha} \left( 1 - \frac{\alpha(k+1+\alpha(k+2))^2}{4((1-\alpha)A_k + \alpha^2(k+1)(k+2))} \right) \label{def:ak1minus}.
\end{align}
Then \eqref{eqn:Sk} can be rewritten as
\begin{align*}
    \mathbf{S}_k = \begin{bmatrix}
        s_{11} & s_{12} & 0 \\
        s_{12} & s_{22} & s_{23} \\
        0      & s_{23} & s_{33}
    \end{bmatrix},
\end{align*}
where
\begin{align}
    \label{eqn:s11}
    & s_{11} = \frac{\left( \alpha(k+1)(k+2) - 2A_k\right)  \left( 2(1-\alpha)A_k + \alpha^2 (k+1)(k+2) - \alpha^3 (k+2)^2 \right)}
   { 2 \left(\alpha(k+2)(k+1-k\alpha) - 2(1-\alpha) A_k \right) } \\
   \label{eqn:s12}
   & s_{12} = - \frac{\alpha (1-\alpha) (k+2) (k+1+\alpha(k+2)) (\alpha (k+1)(k+2)-2A_k)} { 2 \left(\alpha(k+2)(k+1-k\alpha) - 2(1-\alpha) A_k \right) }
\end{align}
\begin{align}
    & s_{22} = \frac{ (1-\alpha^2) (k+2)^2  \left( 2(1-\alpha)A_k - \alpha(k+1)(k+1-\alpha(k+2)) \right) }
    { 2 \left(\alpha(k+2)(k+1-k\alpha) - 2(1-\alpha) A_k \right) } \label{eqn:s22} \\
    & s_{23} = - \frac{(k+2)^2 \left( 2(1-\alpha^2)A_k - \alpha(k+1)^2 + \alpha^3 k(k+2) \right)} { 2 \left(\alpha(k+2)(k+1-k\alpha) - 2(1-\alpha) A_k \right) } \label{eqn:s23} \\
    & s_{33} = \frac{ (k+2)^2  \left( 2(1-\alpha^2)A_k - \alpha(k+1)^2 + \alpha^3 k(k+2) \right) \left( 2(1-\alpha)A_k + \alpha^2 (k+1)(k+2) - \alpha^3 (k+2)^2 \right) }
    { 4(1-\alpha) \left(\alpha(k+2)(k+1-k\alpha) - 2(1-\alpha) A_k \right) \left( (1-\alpha)A_k + \alpha^2 (k+1)(k+2) \right) }. \label{eqn:s33}
\end{align}
The expressions seem ridiculously complicated, but there are a number of repeating terms.
Let
\begin{align*}
    & E_1 = \alpha (k+2)(k+1-k\alpha) - 2(1-\alpha)A_k \\
    & E_2 = \alpha (k+1)(k+2) - 2A_k.
\end{align*}
Because $A_k \le \frac{\alpha(k+1)(k+2)}{2} < u_k$ (see \eqref{eqn:interval}), we have $E_1 > 0, E_2 \ge 0$.
(Note that $E_2 = 0$ only in the boundary case $A_k = \sup I_k^-$.)
Next, put
\begin{align*}
    E_3 = 2(1-\alpha)A_k - \alpha (k+1)(k+1-\alpha(k+2)),
\end{align*}
which is a factor that appears within the definition of $\tau_k$ \eqref{def:tau1}; we have already seen that $E_3 > 0$.
Further, let
\begin{align*}
    & E_4 = 2(1-\alpha)A_k + \alpha^2 (k+1)(k+2) - \alpha^3 (k+2)^2 \\ 
    & E_5 = (1-\alpha)A_k + \alpha^2(k+1)(k+2) \\ 
    & E_6 = 2(1-\alpha^2)A_k - \alpha(k+1)^2 + \alpha^3 k(k+2) \\ 
    & E_7 = k+1 + \alpha(k+2). 
\end{align*}
It is obvious that $E_5, E_7 > 0$, and $E_6 > 0$ follows directly from \eqref{eqn:tau_comparison}.
To see that $E_4 > 0$, observe that $k+1 - \alpha(k+2) = (k+2) \left(\frac{k+1}{k+2} - \alpha\right) \geq (k+2) \left(\frac{1}{2}-\alpha\right) \geq 0$, provided that $\alpha \leq \frac{1}{2}$.
This implies
\begin{align*}
    E_4 = 2(1-\alpha)A_k + \alpha^2 (k+2) \left(k+1 - (k+2)\alpha\right) > 0.
\end{align*}
Now we can rewrite \eqref{eqn:s11} through \eqref{eqn:s33} as
\begin{align*}
    & s_{11} = \frac{E_2 E_4}{2E_1} \\
    & s_{12} = -\frac{\alpha (1-\alpha) (k+2) E_2 E_7}{2E_1} \\
    & s_{22} = \frac{(1-\alpha^2) (k+2)^2 E_3}{2E_1} \\
    & s_{23} = -\frac{(k+2)^2 E_6}{2E_1} \\
    & s_{33} = \frac{(k+2)^2 E_4 E_6}{4(1-\alpha) E_1 E_5}.
\end{align*}
This immediately shows that the diagonal entries $s_{ii}$ are nonnegative for $i=1,2,3$.
By brute-force calculation, it is not difficult to verify the identity
\begin{align*}
    (1+\alpha) E_3 E_4 = \alpha^2 (1-\alpha) E_2 E_7^2 + 2 E_5 E_6.
\end{align*}
Using this, we see that $\mathbf{v}:=
\begin{bmatrix}
\frac{\alpha (k+2) E_7}{2E_5} & \frac{E_4}{2(1-\alpha)E_5} & 1
\end{bmatrix}^\intercal$ satisfies $\mathbf{S}_k \mathbf{v} = 0$, and this implies $\det \mathbf{S}_k = 0$.
The cofactor-expansion of $\det \mathbf{S}_k$ along the first row gives
\begin{align*}
    0 = \det \mathbf{S}_k = s_{11} \begin{vmatrix}
        s_{22} & s_{23} \\ s_{23} & s_{33}
    \end{vmatrix} - s_{12} \begin{vmatrix}
        s_{12} & s_{23} \\ 0 & s_{33}
    \end{vmatrix}
    \iff \begin{vmatrix}
        s_{22} & s_{23} \\ s_{23} & s_{33}
    \end{vmatrix} = \frac{s_{12}^2 s_{33}}{s_{11}} > 0 
\end{align*}
when $s_{11} > 0$, and via continuity argument we can argue that $\begin{vmatrix}
    s_{22} & s_{23} \\ s_{23} & s_{33}
\end{vmatrix} \ge 0$ even in the boundary case $s_{11} = 0$.
Similarly one can show that $\begin{vmatrix}
    s_{11} & s_{12} \\ s_{12} & s_{22}
\end{vmatrix} \ge 0$.
Therefore, we have shown that all diagonal submatrices of $\mathbf{S}_k$ (including the trivial case $\begin{vmatrix}
    s_{11} & 0 \\ 0 & s_{33}
\end{vmatrix} = s_{11}s_{33} \ge 0$) have nonnegative determinants, that is, $\mathbf{S}_k \succeq \bO$.

Finally, \eqref{def:ak1minus} shows that $A_{k+1}$ is increasing with respect to $A_k$.
We see that
\begin{align}
    \label{eqn:akp1_bdy}
    A_{k+1}\Big|_{A_k = \frac{ \alpha(k+1)(k+2) }{2}} = \frac{\alpha(k+2)((k+1)(k+3)-\alpha^2(k+2)^2)}{2(1-\alpha^2)(k+1)}
    < \frac{\alpha (k+2) (k+3)}{2}
\end{align}
and
\begin{align*}
    A_{k+1}|_{A_k=\ell_k} - \ell_{k+1} = 
    \frac{\alpha ^2 \left(  (1-3\alpha-\alpha^2-\alpha^3) k + 1 -8\alpha + \alpha^2 - 2\alpha^3  \right)}
   {2 (1 - \alpha^2) \left( (1+\alpha)^2 k+1+\alpha +2 \alpha ^2 \right)},
\end{align*}
and the last expression is nonnegative because of the assumption \eqref{eqn:EAG-C-alpha-restriction}, which we restate here for the case $R=1$ for convenience: $1-3\alpha-\alpha^2-\alpha^3 \ge 0$ and $1-8\alpha+\alpha^2-2\alpha^3 \ge 0$.
This proves that $A_{k+1} \in I_{k+1}^- \subset I_{k+1}$, as desired.

\noindent
\textbf{Case 2.}
Suppose that $A_k \in I_k^+$. The proof would be similar to Case 1, but choices of $\tau_k$ and $A_{k+1}$ are different.
We let
\begin{align}
    \label{def:tau2}
    \tau_k = \frac{(k+2)^2 \left( 2(1+\alpha)A_k - \alpha(k+1)(k+1+\alpha(k+2))\right)} {4(1+\alpha)A_k - 2\alpha(k+2)(k+1+k\alpha)}.
\end{align}
Since $A_k > \ell_k > \frac{\alpha(k+1)(k+1+\alpha(k+2))}{2(1+\alpha)}$, the denominator and numerator of (\ref{def:tau2}) are both positive and thus $\tau_k > 0$.
Next, let
\begin{align}
    A_{k+1} &= \frac{\alpha(k+2)^2 \left( 4(1+\alpha) A_k - \alpha(k+1+\alpha(k+2))^2 \right)} {4(1+\alpha) \left( (1+\alpha)A_k - \alpha^2(k+1)(k+2) \right)} \nonumber \\
    &= \frac{\alpha (k+2)^2}{1+\alpha} \left( 1 - \frac{\alpha(k+1-\alpha(k+2))^2}{4((1+\alpha)A_k - \alpha^2(k+1)(k+2))} \right) \label{def:ak1plus}.
\end{align}
Then we can check that
\begin{align*}
    & s_{11} = \frac{(2A_k - \alpha(k+1)(k+2))  (2(1+\alpha)A_k - \alpha^2(k+1)(k+2)-\alpha^3 (k+2)^2)}{4(1+\alpha) A_k - 2\alpha(k+2)(k+1+k\alpha)} \\
    & s_{33} = \frac{(k+2)^2 \left( 2(1+\alpha)A_k - \alpha^2 (k+1)(k+2) - \alpha^3 (k+2)^2 \right) \left( 2(1-\alpha^2)A_k - \alpha(k+1)^2 + \alpha^3 k(k+2) \right) }
    {4(1+\alpha) \left( 2(1+\alpha) A_k - \alpha(k+2)(k+1+k\alpha) \right) \left( 2(1+\alpha)A_k - \alpha^2 (k+1)(k+2) \right) },
\end{align*}
and so on.
(Note that $2A_k - \alpha(k+1)(k+2) \geq 0$ because now we are assuming that $A_k \in I_k^+$.)
We omit further details of calculations, but with the above choices of $\tau_k$ and $A_{k+1}$ it can be shown that $\det \mathbf{S}_k = 0$ and $s_{11}, s_{33} \ge 0$, using \eqref{eqn:tau1_positivity} through \eqref{tau1_upperbound}.
As in Case 1, this implies $\mathbf{S}_k \succeq \bO$.

The identity \eqref{def:ak1plus} shows that $A_{k+1}$ is increasing with respect to $A_k$.
Interestingly, although \eqref{def:ak1minus} and \eqref{def:ak1plus} have distinct forms, for the boundary value $A_k = \frac{\alpha (k+1)(k+2)}{2}$, they evaluate to the same expression \eqref{eqn:akp1_bdy} and thus arguments from Case 1 readily show that $A_{k+1} > \ell_{k+1}$.
On the other hand, we have
\begin{align*}
    u_{k+1} - A_{k+1}|_{A_k = u_k} = \frac{\alpha ^2 \left( \left( 1+3\alpha -\alpha ^2 + \alpha ^3 \right)k + 1 + 8 \alpha + \alpha ^2 + 2 \alpha ^3 \right)}
   {2 (1-\alpha^2) \left((1-\alpha)^2 k+1-\alpha + 2\alpha^2\right)}
\end{align*}
and the last term is positive for any $\alpha \in (0,1)$, i.e., $A_{k+1} < u_{k+1}$.
This completes Case 2.

\textbf{Proof of the theorem statement.}
Given that $A_k \in I_k^-$ implies $A_{k+1} \in I_{k+1}^-$ (which has been proved in Case 1), the rest is easy.
If we take $A_0 = \ell_0 = \frac{\alpha}{1+\alpha}$, then because $\mathbf{S}_k \succeq \bO$ for all $k\ge 0$, we see that $V_k$ is nonincreasing:
\begin{align*}
    \frac{\alpha}{1+\alpha} \|\bz^0 - \bz^\star\|^2 & \geq \frac{\alpha}{1+\alpha} \left\|\sop(\bz^0) \right\|^2 = V_0 \geq \cdots \geq V_k = A_k \left\|\sop(\bz^k) \right\|^2 + (k+1) \left\langle \bz^k - \bz^0, \sop(\bz^k) \right\rangle,
\end{align*}
where the first inequality follows from Lipschitzness of $\sop$ (recall that we are assuming that $R=1$).
Also by \eqref{eqn:interval} and \eqref{eqn:tau1_positivity},
\begin{align}
    \label{eqn:ell_k_lowerbound}
    A_k \geq \ell_k > \frac{\alpha (k+1) (k+1+\alpha(k+2))}{2(1+\alpha)} = \frac{\alpha (k+1)}{2} \frac{(1+\alpha)(k+1) + \alpha}{1+\alpha} > \frac{\alpha (k+1)^2}{2}.
\end{align}
Hence, we obtain
\begin{align*}
    \frac{\alpha}{1+\alpha}\|\bz^0-\bz^\star\|^2 \geq V_k & \geq \ell_k \left\|\sop(\bz^k) \right\|^2 + (k+1)\left\langle \bz^k - \bz^0, \sop(\bz^k) \right\rangle \\
    & \labelrel\geq{ineq:EAG-C-monotonicity} \frac{\alpha(k+1)^2}{2} \left\|\sop(\bz^k) \right\|^2 + (k+1) \left\langle \bz^\star - \bz^0, \sop(\bz^k) \right\rangle\\
    & \labelrel\geq{ineq:EAG-C-Young} \frac{\alpha(k+1)^2}{2} \left\|\sop(\bz^k) \right\|^2 - (k+1) \left( \frac{1}{\alpha(k+1)} \|\bz^\star - \bz^0\|^2 + \frac{\alpha(k+1)}{4} \left\|\sop(\bz^k) \right\|^2 \right),
\end{align*}
where \eqref{ineq:EAG-C-monotonicity} follows from \eqref{eqn:ell_k_lowerbound} and the monotonicity inequality $\langle \bz^k - \bz^\star, \sop(\bz^k) \rangle \ge 0$, and \eqref{ineq:EAG-C-Young} follows from Young's inequality.
Rearranging terms, we conclude that
\begin{align*}
\left\|\sop(\bz^k) \right\|^2 \leq  \frac{4}{\alpha(k+1)^2} \left( \frac{\alpha}{1+\alpha} + \frac{1}{\alpha} \right) \|\bz^0 - \bz^\star\|^2 = \frac{C\|\bz^0 - \bz^\star\|^2}{(k+1)^2},
\end{align*}
where $C = \frac{4(1+\alpha+\alpha^2)}{\alpha^2 (1+\alpha)}$.

\textbf{Proof of Lemma~\ref{lemma:lkuk}.}
Direct calculation gives
\begin{gather*}
    u_k - \frac{\alpha(k+1)(k+2)}{2} = \frac{\alpha^2 (k+2)}{2(1-\alpha)} > 0\\
    \frac{\alpha(k+1)(k+2)}{2} - \ell_k = \frac{\alpha^2 (k+2)}{2(1+\alpha)} > 0,
\end{gather*}
showing \eqref{eqn:interval}.
Next,
\begin{align*}
    \ell_k - \frac{\alpha(k+1)(k+1+\alpha(k+2))}{2(1+\alpha)} = \frac{\alpha(k+1-\alpha(k+2))}{2(1+\alpha)} \geq 0
\end{align*}
because $k+1 - \alpha(k+2) = (k+2) (\frac{k+1}{k+2} - \alpha) \geq (k+2) (\frac{1}{2}-\alpha) \geq 0$, which shows \eqref{eqn:tau1_positivity}.
Similarly, we observe that
\begin{align*}
    \frac{\alpha(k+1)(k+1+\alpha(k+2))}{2(1+\alpha)} - \frac{\alpha(k+1)^2 - \alpha^3 k(k+2)}{2(1-\alpha^2)} = \frac{\alpha^2 (k+1-\alpha(k+2))}{2(1-\alpha^2)} & \geq 0 \\
    \frac{\alpha(k+1)^2 - \alpha^3 k(k+2)}{2(1-\alpha^2)} - \frac{\alpha(k+1)(k+1-\alpha(k+2))}{2(1-\alpha)} = \frac{\alpha^2 (k+1+\alpha(k+2))}{2(1-\alpha^2)} & > 0 \\
    \frac{\alpha(k+1)^2 - \alpha^3 k(k+2)}{2(1-\alpha^2)} - \frac{\alpha^2 (k+1)(k+2)}{1+\alpha} = \frac{\alpha(k+1-\alpha(k+2))^2}{2(1-\alpha^2)} & \geq 0 \\
    \frac{\alpha^2 (k+1)(k+2)}{1+\alpha} - \frac{\alpha^2 (k+1)(k+2) + \alpha^3 (k+2)^2}{2(1+\alpha)} = \frac{\alpha^2 (k+2)(k+1-\alpha(k+2))}{2(1+\alpha)} & \geq 0,
\end{align*}
and each line corresponds to an inequality within \eqref{eqn:tau_comparison}, \eqref{tau2_positivity} and \eqref{tau1_upperbound}.

\section{Omitted proofs of Section~\ref{sec:lower-bound}}
In this section, we provide a self-contained discussion on the complexity lower bound results for linear operator equations from \citet{nemirovsky1991optimality, nemirovsky1992information}.

\subsection{Proof of Theorem~\ref{thm:lowerbound}}
\label{section:Proof_Theorem_3}
The proof of Theorem~\ref{thm:lowerbound} was essentially completed in the main body of the paper, except the argument regarding translation, \eqref{eqn:alg_span_expanded}, and the proof of Lemma~\ref{lemma:Nemirovsky}.

We first provide the precise meaning of the translation invariance that we are to prove.
Given a saddle function $\lagrange$ and $\bz \in \reals^n \times \reals^n$, let $\bz_\lagrange^\star (\bz)$ be the saddle point of $\lagrange$ nearest to $\bz$.
For any $\bz^0 \in \reals^n \times \reals^n$, $k \ge 0$ and $D > 0$, define
\begin{align*}
    \mathfrak{T}\left(\bz^0; k,D\right) := \left\{ \bz^k \,\, \Bigg|\,\,
    \begin{aligned}
        & \lagrange(\bx,\by) = \langle \bA\bx-\bb, \by-\bc \rangle, \,\, \bA \in \reals^{n\times n}, \,\, \bb,\bc \in \reals^n, \,\, \left\|\bz_\lagrange^\star \left(\bz^0\right) - \bz^0\right\| \le D, \\
        & \bz^j = \cA(\bz^0, \dots, \bz^{j-1}; \lagrange),\,\, j=1,\dots,k, \,\, \cA \in \fA_{\textrm{sep}}
    \end{aligned}
    \right\}.
\end{align*}
We will show that
\begin{align*}
    \mathfrak{T} \left(\bz^0; k,D\right) = \bz^0 + \mathfrak{T} \left(0; k,D \right)
\end{align*}
holds for any $\bz^0 \in \reals^n \times \reals^n$.

Let $\bz^0 = (\bx^0, \by^0)$ and $\lagrange (\bx,\by) = \langle \bA\bx-\bb, \by-\bc \rangle$ be given, and assume that $\|\bz_\lagrange^\star (\bz^0) - \bz^0\| \le D$.
Let $\bb_0 = \bb - \bA\bx^0$ and $\bc_0 = \bc-\by^0$. Then
\begin{align*}
    & \nabla_\bx \lagrange_0 (\bx^0,\by^0) = \bA^\intercal (\by^0-\bc) = -\bA^\intercal \bc_0 \\
    & \nabla_\by \lagrange_0 (\bx^0,\by^0) = \bA\bx^0 - \bb = -\bb_0.
\end{align*}
Hence, \eqref{eqn:alg-span} with $k=1$ reads as
\begin{align*}
    & \bx^1 - \bx^0 \in \spann \{\bA^\intercal \bc_0\} \stackrel{\Delta}{=} \mathcal{X}_1 (\bA; \bb_0,\bc_0) \\
    & \by^1 - \by^0 \in \spann \{\bb_0\} \stackrel{\Delta}{=} \mathcal{Y}_1 (\bA; \bb_0,\bc_0).
\end{align*}
This further shows that
\begin{align*}
    & \nabla_\bx \lagrange_0 (\bx^1,\by^1) = \bA^\intercal (\by^1-\bc) = \bA^\intercal (\by^1-\by^0) - \bA^\intercal \bc_0 \in \spann \{ \bA^\intercal \bb_0, \bA^\intercal \bc_0 \} \\
    & \nabla_\by \lagrange_0 (\bx^1,\by^1) = \bA\bx^1 - \bb = \bA(\bx^1-\bx^0) - \bb_0 \in \spann\{ \bA(\bA^\intercal \bc_0), \bb_0 \},
\end{align*}
and \eqref{eqn:alg-span} with $k=2$ becomes
\begin{align*}
    & \bx^2 - \bx^0 \in \spann\{ \bA^\intercal \bc_0, \bA^\intercal \bb_0 \} \stackrel{\Delta}{=} \mathcal{X}_2 (\bA;\bb_0,\bc_0)\\
    & \by^2 - \by^0 \in \spann\{ \bb_0 , \bA\bA^\intercal \bc_0 \} \stackrel{\Delta}{=} \mathcal{Y}_2 (\bA;\bb_0,\bc_0).
\end{align*}
As one can see, we have $\bx^k-\bx^0 \in \mathcal{X}_k (\bA;\bb_0,\bc_0)$ and $\by^k-\by^0 \in \mathcal{Y}_k (\bA;\bb_0,\bc_0)$, where we inductively define
\begin{align*}
    & \mathcal{X}_{k+1} (\bA;\bb_0,\bc_0) = \spann\{\bA^\intercal \bc_0\} + \bA^\intercal \mathcal{Y}_{k} (\bA;\bb_0,\bc_0)\\ 
    & \mathcal{Y}_{k+1} (\bA;\bb_0,\bc_0) = \spann\{\bb_0\} + \bA\mathcal{X}_k (\bA;\bb_0,\bc_0).
\end{align*}
Then it is not difficult to see that for $k\ge 2$,
\begin{gather*}
    \mathcal{X}_k (\bA;\bb_0,\bc_0) = \spann \left\{ \bA^\intercal \bc_0, \bA^\intercal(\bA \bA^\intercal) \bc_0, \dots, \bA^\intercal (\bA \bA^\intercal)^{\lfloor \frac{k-1}{2} \rfloor} \bc_0 \right\} + \spann \left\{ \bA^\intercal \bb_0, \bA^\intercal (\bA\bA^\intercal)\bb_0, \dots, \bA^\intercal (\bA\bA^\intercal)^{\lfloor \frac{k}{2} \rfloor - 1} \bb_0 \right\}\\
    \mathcal{Y}_k (\bA;\bb_0,\bc_0) = \spann \left\{ \bb_0, (\bA\bA^\intercal)\bb_0, \dots, (\bA\bA^\intercal)^{\lfloor \frac{k-1}{2} \rfloor} \bb_0 \right\} + \spann \left\{ \bA\bA^\intercal \bc_0, \dots, (\bA\bA^\intercal)^{\lfloor \frac{k}{2} \rfloor} \bc_0 \right\}.
\end{gather*}
Now consider $\lagrange_0 (\bx,\by) := \langle \bA\bx - \bb_0 , \by-\bc_0 \rangle = \left\langle \bA(\bx+\bx^0) - \bb, \by+\by^0 - \bc \right\rangle$.
Because $\bz^\star_{\lagrange_0}$ is a saddle point of $\lagrange_0$ if and only if $\bz^\star_{\lagrange_0} + \bz^0$ is a saddle point of $\lagrange$, we have $\bz^\star_{\lagrange_0}(0) = \bz^\star_\lagrange (\bz^0) - \bz^0$, and thus $\|\bz^\star_{\lagrange_0}(0)\| \le D$.
Therefore, if we let
\begin{align*}
    \mathcal{S} (\bA; D) \stackrel{\Delta}{=} \left\{ (\Tilde{\bb},\Tilde{\bc}) \in \reals^n \times \reals^n \,\bigg|\, \left\|\bz^\star_{\Tilde{\lagrange}}(0)\right\| \le D , \textnormal{ where } \Tilde{\lagrange}(\bx,\by) = \langle \bA\bx-\Tilde{\bb}, \by-\Tilde{\bc} \rangle \right\},
\end{align*}
then
\begin{align*}
    \mathfrak{T}\left( \bz^0; k,D \right) = \bigcup_{\substack{\bA\in\reals^{n\times n}\\(\bb_0,\bc_0)\in\mathcal{S}(\bA;D)}} \bz^0 + \left(\mathcal{X}_k (\bA;\bb_0,\bc_0) \times \mathcal{Y}_k (\bA;\bb_0,\bc_0) \right).
\end{align*}
This proves that the translation invariance holds with $\mathfrak{T} (0;k,D) = \bigcup_{\substack{\bA\in\reals^{n\times n}\\(\bb_0,\bc_0)\in\mathcal{S}(\bA;D)}} \left(\mathcal{X}_k (\bA;\bb_0,\bc_0) \times \mathcal{Y}_k (\bA;\bb_0,\bc_0) \right)$ and in particular, shows \eqref{eqn:alg_span_expanded}.

\subsection{Complexity of solving linear operator equations and minimax polynomials}
\label{section:minimax_poly}
We first make some general observations.
Suppose that we are given a symmetric matrix $\bA \in \reals^{n \times n}$, $\bb \in \reals^n$, and an integer $k \ge 1$.
Then any $\bx \in \mathcal{K}_{k-1}(\bA;\bb) = \spann \{\bb, \bA\bb, \dots, \bA^{k-1}\bb\}$ can be expressed in the form
\begin{align*}
    \bx = q(\bA) \bb, \quad \text{where } q(t) = q_0 + q_1 t + \cdots + q_{k-1} t^{k-1},
\end{align*}
for some $q_0, \dots, q_{k-1} \in \reals$.
Then we can write
\begin{align}
    \label{eqn:residual_poly}
    \bb - \bA\bx = \bb - \bA q(\bA) \bb = (\bI - \bA q(A)) \bb = p(\bA) \bb,
\end{align}
where $p(t) = 1-tq(t)$ is a polynomial of degree at most $k$ satisfying $p(0) = 1$.
Note that conversely, given any polynomial $\Tilde{p}(t)$ with degree $\le k$ and constant term $1$, one can decompose it as $\Tilde{p}(t) = 1 - t\Tilde{q}(t)$ and recover a polynomial $\Tilde{q}$ of degree $\le k-1$ corresponding to $\bx$.

Now suppose further there exists $\bx^\star \in \reals^n$ such that $\bb = \bA\bx^\star$ and $\|\bx^\star\| \le D$.
The symmetric matrix $\bA$ has an orthonormal eigenbasis $\bv_1, \dots, \bv_n$, corresponding to eigenvalues $\lambda_1, \dots, \lambda_n$, so we can write $\bx^\star = c_1 \bv_1 + \cdots + c_n \bv_n$ for some $c_1,\dots,c_n \in \reals$.
Using \eqref{eqn:residual_poly}, we obtain
\begin{equation}
    \label{eqn:poly_upper_bound}
    \begin{aligned}
        \left\|\bA\bx-\bb \right\|^2 = \left\| p(\bA) \bA\bx^\star \right\|^2 = \left\| \sum_{j=1}^n c_j \bA p(\bA) \bv_j \right\|^2 &= \left\| \sum_{j=1}^n c_j \lambda_j p(\lambda_j) \bv_j \right\|^2 \\
    &= \sum_{j=1}^n c_j^2 \lambda_j^2 p(\lambda_j)^2 \le D^2 \left(\max_{j=1,\dots,n} \lambda_j^2 p(\lambda_j)^2 \right).
    \end{aligned}
\end{equation}
We define the problem class by $\|\bA\| \le R$, which is equivalent to $\lambda_j \in [-R,R]$ for all $j=1,\dots,n$.
Therefore, we consider a method corresponding to a polynomial $q(t)$ such that $p(t) = 1-tq(t)$ minimizes
\begin{align*}
    \max_{\lambda \in [-R,R]} \lambda^2 p(\lambda)^2 = \left( \max_{\lambda \in [-R,R]} \left| \lambda p(\lambda) \right| \right)^2.
\end{align*}
More precisely, if $p_k^\star (t) = 1 - tq_k^\star(t)$ minimizes the last quantity among all $p(t)$ such that $\deg p \le k$ and $p(0) = 1$, and if we put $\bx^k = q_k^\star (\bA) \bb$, then \eqref{eqn:poly_upper_bound} implies
\begin{gather}
    \nonumber
    \left\|\bA\bx^k-\bb \right\|^2 = \sum_{j=1}^n c_j^2 \lambda_j^2 \left(p_k^\star(\lambda_j)\right)^2 \le D^2 M^\star(k,R)^2 \\
    \label{eqn:minimax_polynomial}
    M^\star(k,R) \stackrel{\Delta}{=} \min_{\substack{\deg p \le k \\ p(0) = 1}} \max_{\lambda \in [-R,R]} |\lambda p(\lambda)|,
\end{gather}
for all $\bA$ whose spectrum belongs to $[-R,R]$ and $\bb = \bA\bx^\star$ with $\|\bx^\star\| \le D$.
As $p_k^\star$ solves \eqref{eqn:minimax_polynomial}, it is called a \emph{minimax polynomial}.

In order to establish Lemma~\ref{lemma:Nemirovsky}, we present a two-fold analysis in the following.
First, we compute the quantity \eqref{eqn:minimax_polynomial} by explicitly naming $p_k^\star$ for each $k \ge 1$. (This was given by \citet{nemirovsky1992information}, but without a proof.)
Then, following the exposition from \citep{nemirovsky1991optimality}, we show that there exists an instance of $(\bA,\bb)$ such that
\begin{align*}
    \|\bA q(\bA) \bb - \bb\|^2 \ge D^2 M^\star (k,R)^2
\end{align*}
holds for any polynomial $q$ of degree $\le k-1$.

\subsection{Proof of Lemma~\ref{lemma:Nemirovsky}}

The solutions to \eqref{eqn:minimax_polynomial} are characterized using the \emph{Chebyshev polynomials of first kind}, defined by
\begin{align*}
    T_N (\cos \theta) = \cos (N\theta), \quad N \geq 1,
\end{align*}
or equivalently by $T_N (t) = \cos (N \arccos t)$.
If $N=2d$ for some nonnegative integer $d$, then $T_N$ is an even polynomial satisfying $T_N (0) = \cos (d \pi) = (-1)^d$.
On the other hand, if $N = 2d+1$, then $T_N$ is an odd polynomial of the form
\begin{align}
    \label{eqn:Chebyshev_odd}
    T_{2d+1}(t) = (-1)^d (2d+1) t + \cdots,
\end{align}
which can be shown via induction using the recurrence relation $T_{N+1}(t) = 2t T_{N}(t) - T_{N-1}(t)$, which follows from the trigonometric identity
\begin{align*}
    \cos ((N+1)\theta) + \cos ((N-1)\theta) = 2 \cos(N\theta) \cos \theta.
\end{align*}
Based on arguments from \citep{nemirovsky1992information, mason2002chebyshev}, we will show that given $k\ge 1$ and $m:= \lfloor\frac{k}{2}\rfloor$,
\begin{align*}
    p_k^\star (t) := \frac{(-1)^m}{2m+1} \left(\frac{R}{t}\right) T_{2m+1} \left(\frac{t}{R}\right)
\end{align*}
solves \eqref{eqn:minimax_polynomial}.

The Chebyshev polynomials satisfy the \emph{equioscillation property} which makes them so special: the extrema of $T_N$ within $[-1,1]$ occur at $t_j = \cos \frac{(N-j) \pi}{N}$ for $j=0,\dots,N$, and the signs of the extremal values alternate.
Indeed, we have $|T_N(t) = \cos(N \arccos t)| \le 1$ for all $t \in [-1,1]$, and for each $j=0,\dots,N$,
\begin{align*}
    T_N(t_j) = \cos \left( N \frac{(N-j)\pi}{N} \right) = \cos (N-j)\pi = (-1)^{N-j}.
\end{align*}
Also, we have $T_N(t_{j}) = -T_N (t_{j-1})$ for each $j=1,\dots,n$.

Given $k \ge 1$, we denote by $\mathcal{P}_k$ the collection of all polynomials $p$ of degree $\leq k$ with $p(0)=1$.
Recall that we are to minimize
\begin{align}
    \label{eqn:poly_max_over_interval}
    M(p,R) := \max_{\lambda\in [-R,R]} |\lambda \, p(\lambda)|
\end{align}
over $p\in \mathcal{P}_k$.
If $p \in \mathcal{P}_k$ minimizes \eqref{eqn:poly_max_over_interval}, then so does $p_\even (t) := \frac{p(t) + p(-t)}{2}$, since for all $\lambda \in [-R,R]$
\begin{align}
    \label{eqn:replace_by_even}
    |\lambda p_{\textrm{ev}}(\lambda)| = |\lambda| \cdot \left|\frac{p(\lambda) + p(-\lambda)}{2} \right| \leq \frac{|\lambda p(\lambda)|}{2} + \frac{|(-\lambda) p(-\lambda)|}{2} \leq \frac{M(p,R)}{2} + \frac{M(p,R)}{2} = M(p,R)
\end{align}
holds, which implies that $M(p_\even,R) \le M(p,R)$.

Observe that $p_k^\star \in \mathcal{P}_k$ due to \eqref{eqn:Chebyshev_odd}.
Next, note that $\lambda p_k^\star (\lambda) = \frac{(-1)^m R}{2m+1} T_{2m+1}(\frac{\lambda}{R})$ has extrema of alternating signs and same magnitude within $[-R,R]$, which occur precisely at $\lambda_j := R \cos \frac{(2m+1-j)\pi}{2m+1}$, where $j=0,\dots,2m+1$.
Suppose that $p_k^\star$ is not a minimizer of $M(p,R)$ over $\mathcal{P}_k$, so that there exists $p\in \mathcal{P}_k$ such that
\begin{align}
    \label{eqn:minimax_assume_contrary}
    |\lambda_j p (\lambda_j)| \le M(p,R) < M(p_k^\star, R) = |\lambda_j p_k^\star (\lambda_j)| \quad (j=0,\dots,2m+1).
\end{align}
Due to \eqref{eqn:replace_by_even}, by replacing $p$ with $p_{\mathrm{ev}}$ if necessary, we may assume that $p$ is even and has degree $\le 2m$.
Since $\lambda_j \ne 0$ for all $j=0,\dots,2m+1$, the condition \eqref{eqn:minimax_assume_contrary} reduces to $|p(\lambda_j)| < |p_k^\star(\lambda_j)|$.

As $p$ and $p_k^\star$ are both polynomials of degree $\leq 2m$ and constant terms 1, we can write
\begin{align*}
    p_k^\star (\lambda) - p(\lambda) = \lambda q(\lambda)
\end{align*}
for some polynomial $q$ of degree $\leq 2m-1$.
But then $|p(\lambda_j)| = |p_k^\star (\lambda_j) - \lambda_j q(\lambda_j)| < |p_k^\star(\lambda_j)|$, which implies that $p_k^\star (\lambda_j)$ and $\lambda_j q(\lambda_j)$ have same signs for $j=0,\dots,2m+1$.
Now, because $p_k^\star (\lambda_j)$ have alternating signs and
\[
\lambda_0 < \cdots < \lambda_m < 0 < \lambda_{m+1} < \cdots < \lambda_{2m+1},
\]
we see that the signs of $q(\lambda_j)$ alternate over $j=0,\dots,m$ and over $j=m+1,\dots,2m+1$, respectively.
Therefore, $q$ must have at least one zero in each open interval $(\lambda_j, \lambda_{j+1})$ for $j=0,\dots,m-1,m+1,\dots,2m$.
This implies that $q(t) \equiv 0$ since $\deg q \leq 2m-1$, while $q$ has at least $2m$ zeros.
Therefore, we arrive at $p_k^\star = p$, which is a contradiction.

We have established that
\begin{align}
    \label{eqn:M_star_exact_value}
    M^\star(k,R) = M(p_k^\star, R) = \left| \lambda_j p_k^\star (\lambda_j) \right| = \frac{R}{2m+1} = \frac{R}{2\lfloor k/2 \rfloor +1} \, \quad (j=0,\dots,2m+1).
\end{align}
Furthermore, the above arguments show that the minimization of \eqref{eqn:poly_max_over_interval} over $p \in \mathcal{P}_k$ is in fact the same as the minimization of
\begin{align}
    \label{eqn:poly_max_discrete}
    \max_{j=0,\dots,2m+1} \left| \lambda_j p(\lambda_j) \right| = \max_{\lambda \in \Lambda} |\lambda p(\lambda)|, \quad \Lambda := \{\lambda_0,\lambda_1,\dots,\lambda_{2m+1}\}.
\end{align}
Note that the trick of replacing $p$ by $p_\even$ is still applicable to \eqref{eqn:poly_max_discrete}, but only because the set $\Lambda$ is symmetric with respect to the origin.
Now we can write
\begin{align}
    \label{eqn:poly_min_max_discrete_squared}
    M^\star(k,R)^2 = \left( \min_{p \in \mathcal{P}_k}\max_{\lambda \in [-R,R]} |\lambda p(\lambda)| \right)^2 = \left( \min_{p \in \mathcal{P}_k}\max_{\lambda \in \Lambda} |\lambda p(\lambda)| \right)^2
    = \min_{p  \in \mathcal{P}_k} \max_{\lambda \in \Lambda} \lambda^2 p(\lambda)^2,
\end{align}
and the final problem from the line \eqref{eqn:poly_min_max_discrete_squared} is equivalent to
\begin{align}
    \label{eqn:poly_min_max_SOCP_form}
    \begin{array}{ll}
        \underset{\nu \in \reals,\, p \in \mathcal{P}_k}{\mbox{minimize}} & \nu \\
        \mbox{subject to} & \lambda_j^2 p(\lambda_j)^2 \le \nu, \quad j=0,\dots,2m+1.   
    \end{array}
\end{align}
We can identify any $p(t) = 1 + p_1 t + \cdots + p_k t^k \in \mathcal{P}_k$ as the vector $(p_1,\dots,p_k) \in \reals^k$.
Under this identification, \eqref{eqn:poly_min_max_SOCP_form} is a second order cone program (as the constraints are convex quadratic in $p_1,\dots,p_k$), and Slater's constraint qualification is clearly satisfied.
Hence $M^\star(k,R)^2$ equals the optimal value of the dual problem
\begin{align}
    \label{eqn:poly_min_max_dual}
    \begin{array}{ll}
        \underset{\boldsymbol{\mu} \in \reals^{2m+2}}{\mbox{maximize}} \,\, \underset{p \in \mathcal{P}_k}{\mbox{minimize}}  & \sum_{j=0}^{2m+1} \mu_j \lambda_j^2 p(\lambda_j)^2 \\
        \mbox{subject to} &  \sum_{j=0}^{2m+1} \mu_j = 1, \\
        & \boldsymbol{\mu} \ge 0.
    \end{array}
\end{align}
Let $\boldsymbol{\mu}^\star = (\mu_0^\star, \dots, \mu_{2m+1}^\star)$ be the dual optimal solution to \eqref{eqn:poly_min_max_dual}.
Provided that $n \ge k+2 \ge 2m+2$, we can take standard basis vectors (with $0$-indexing) $\be_0, \dots, \be_{2m+1} \in \reals^n$.
Define $\bA$ by
\begin{align*}
    \bA\be_j = \lambda_j \be_j \quad (j=0,\dots,2m+1), \quad \bA\bv = 0 \quad  (\bv \perp \spann\{\be_0, \dots, \be_{2m+1}\})
\end{align*}
and let
\begin{align*}
    \bb = \bA\bx^\star, \quad \bx^\star = D \sum_{j=0}^{2m+1} \left(\mu_j^\star\right)^{1/2} \be_j
\end{align*}
so that $\|\bx^\star\| = D$.
For any given $\bx = q(\bA) \bb$ with $\deg q \le k-1$, we use \eqref{eqn:poly_upper_bound} to rewrite $\|\bA\bx-\bb\|^2$ as
\begin{align*}
    \|\bA\bx-\bb\|^2 = D^2 \sum_{j=0}^{2m+1} \mu_j^\star \lambda_j^2 \left(1-\lambda_j q(\lambda_j) \right)^2 = D^2 \sum_{j=0}^{2m+1} \mu_j^\star \lambda_j^2 p(\lambda_j)^2,
\end{align*}
where $p(t) = 1- tq(t) \in \mathcal{P}_k$.
But since $(p_k^\star, \boldsymbol{\mu}^\star)$ is the primal-dual solution pair to the problems \eqref{eqn:poly_min_max_SOCP_form} and \eqref{eqn:poly_min_max_dual}, $p_k^\star$ minimizes $\sum_{j=0}^{2m+1} \mu_j^\star \lambda_j^2 p(\lambda_j)^2$ within $\mathcal{P}_k$.
Therefore,
\begin{align*}
\|\bA\bx-\bb\|^2 = D^2 \sum_{j=0}^{2m+1} \mu_j^\star \lambda_j^2 p(\lambda_j)^2 \ge D^2 \sum_{j=0}^{2m+1} \mu_j^\star \lambda_j^2 p_k^\star(\lambda_j)^2 = D^2 M^\star(k,R)^2 = \frac{R^2 D^2}{2(\lfloor k/2 \rfloor+1)^2},
\end{align*}
which establishes \eqref{eqn:lemma_lowerbound}.

\subsection{Proof of Lemma~\ref{lemma:normaleq_chebyshev_algorithm}}
Let $k \ge 0$ be a given (fixed) integer.
Consider the polynomial $p_k^\star$ we defined in the previous section.
It is an even polynomial of degree $2\lfloor \frac{k}{2} \rfloor$, and thus $p_k^\star \left(\sqrt{t}\right)$ is a polynomial in $t$ of degree $\lfloor \frac{k}{2} \rfloor$, whose constant term is $p_k^\star(0) = 1$.
Therefore, we can write $p_k^\star \left(\sqrt{t}\right) = 1 - tq_k(t)$ for some polynomial $q_k$.
We will show that
\begin{align}
    \label{eqn:linear_optimal_algorithm}
    \bz^k = q_k \left(\bB^\intercal \bB\right) \bB^\intercal \bv
\end{align}
satisfies $\|\bB\bz^k - \bv\|^2 \le \frac{R^2 D^2}{2(\lfloor k/2 \rfloor +1)^2}$ for any (possibly non-symmetric) $\bB \in \reals^{m\times m}$ and $\bv = \bB\bz^\star$ satisfying $\|\bB\| \le R$ and $\left\|\bz^\star\right\| \le D$.
The equation \eqref{eqn:linear_optimal_algorithm} defines an algorithm within the class $\fA_{\textrm{lin}}$, as $q_k$ is of degree $\lfloor \frac{k}{2} \rfloor - 1$, so that $\bz^k$ is determined by $2\lfloor \frac{k}{2} \rfloor -1 \le k-1$ queries to the matrix multiplication oracle.

We proceed via arguments similar to derivations in \ref{section:minimax_poly}.
First, observe that
\begin{align}
    \label{eqn:error_via_matrix_sqrt}
    \left\|\bB\bz^k - \bv \right\|^2 = \left\|\bB\bz^k - \bB\bz^\star\right\|^2 = (\bz^k-\bz^\star)^\intercal \bB^\intercal \bB (\bz^k-\bz^\star) = (\bz^k-\bz^\star)^\intercal |\bB|^2 (\bz^k-\bz^\star) = \left\| |\bB|\bz^k - |\bB|\bz^\star \right\|^2,
\end{align}
where $|\bB|$ is the matrix square root of the positive semidefinite matrix $\bB^\intercal \bB$.
Rewriting \eqref{eqn:linear_optimal_algorithm} in terms of $|\bB|$, we obtain
\begin{align*}
    \bz^k = q_k\left(\bB^\intercal \bB \right) \bB^\intercal \bB \bz^\star = q_k \left(|\bB|^2\right) |\bB|^2 \bz^\star.
\end{align*}
Plugging the last equation into \eqref{eqn:error_via_matrix_sqrt} gives
\begin{align*}
    \left\| |\bB|\bz^\star - |\bB|\bz^k \right\|^2 = \left\| \left( \bI - |\bB|^2 q_k \left(|\bB|^2 \right) \right) |\bB| \bz^\star \right\|^2 = \left\| p_k^\star\left(|\bB|\right) |\bB| \bz^\star \right\|^2 .
\end{align*}
Finally, because $|\bB|$ is a symmetric matrix whose eigenvalues are within $[0,R]$, we can apply \eqref{eqn:poly_upper_bound} with $|\bB|,\bz^\star$ in places of $\bA,\bx^\star$, and use \eqref{eqn:M_star_exact_value} to conclude that
\begin{align*}
    \left\| |\bB|\bz^\star - |\bB|\bz^k \right\|^2 \le D^2 \left(\max_{\lambda \in [0,R]} \lambda^2 p_k^\star (\lambda)^2 \right) \le D^2 \left(\max_{\lambda \in [-R,R]} \lambda^2 p_k^\star (\lambda)^2 \right) = \frac{R^2 D^2}{(2\lfloor k/2 \rfloor+1)^2}.
\end{align*}

\subsection{Proof of Theorem 4}
\label{section:remove_span_condition}

We first describe the general class $\fA$ of algorithms without the linear span assumption.
An algorithm $\cA$ within $\fA$ is a sequence of deterministic functions $\cA_1, \cA_2, \dots$, each of which having the form
\begin{align*}
(\bz^i, \overline{\bz}^i) = \cA_i \left(\bz^0, \cO(\bz^0; \lagrange), \dots, \cO(\bz^{i-1};\lagrange); \lagrange \right)
\end{align*}
for $i \ge 1$, where $\bz^0 = (\bx^0,\by^0) \in \reals^n \times \reals^m$ is an initial point and $\mathcal{O}\colon (\reals^n \times \reals^m) \times \cL_R (\reals^n \times \reals^m) \to \reals^n \times \reals^m$ is the gradient oracle defined as
\[
\cO ((\bx,\by); \lagrange) = \left( \nabla_\bx \lagrange (\bx,\by), \nabla_\by \lagrange (\bx,\by) \right).
\]
The sequence $\{\bz^i\}_{i\ge 0}$ are the \emph{inquiry points}, and $\{\overline{\bz}^i\}_{i\ge 0}$ are the \emph{approximate solutions} produced by $\cA$.
When $k \ge 1$ is the predefined maximum number of iterations, then we assume $\overline{\bz}^k = \bz^k$ without loss of generality.
Similar definitions for deterministic algorithms have been considered in \citep{nemirovsky1991optimality, ouyang2019lower}.

To clarify, given $\lagrange \in \cL_R (\reals^n \times \reals^m)$, an algorithm $\cA$ uses only the previous oracle information to choose the next inquiry point and approximate solution.
Therefore, if $\cO(\bz^i;\lagrange_1) = \cO(\bz^i;\lagrange_2)$ for all $i=0,\dots,k-1$, then the algorithm output $(\bz^k, \overline{\bz}^k)$ for the two functions will coincide, even if $\lagrange_1 \ne \lagrange_2$.
In that sense, $\cA$ is \emph{deterministic}, \emph{black-box}, and \emph{gradient-based}.

Now we precisely restate Theorem~\ref{thm:lowerbound_without_span_condition}.

\begin{reptheorem}{thm:lowerbound_without_span_condition}
Let $k \ge 1$ and $n \ge 3k+2$.
Let $\cA \in \fA$ be a deterministic black-box gradient-based algorithm for solving convex-concave minimax problems on $\reals^n \times \reals^n$.
Then for any initial point $\bz^0 \in \reals^n \times \reals^n$, there exists $\lagrange \in \cL_R^{\textnormal{biaff}}(\reals^n \times \reals^n)$ with a saddle point $\bz^\star$, for which $\bz^k$, the $k$-th iterate produced by $\cA$, satisfies
\[
\|\nabla \lagrange (\bz^k)\| \ge \frac{\|\bz^0 - \bz^\star\|^2}{(2\lfloor k/2 \rfloor+1)^2}.
\]
\end{reptheorem}

\begin{proof}
Let $\bz^0 = (\bx^0,\by^0) \in \reals^n \times \reals^n$ be given.
Take $\bA$ and $\bb$ as in Lemma~\ref{lemma:Nemirovsky}.
Denote by $\bx^{\textrm{min}}$ the minimum norm solution to $\bA\bx = \bb$.
Recall the construction of $\bA$ and $\bb$, where $\mathcal{R}(\bA) = \spann\{\be_0,\dots,\be_{2m+1}\} \perp \ker(\bA)$.
Define
\[
\lagrange_0 (\bx^0,\by^0) = -\bb^\intercal (\bx-\bx^0) + (\bx-\bx^0)^\intercal \bA (\by-\by^0) - \bb^\intercal (\by-\by^0).
\]
Then $\left(\nabla_\bx \lagrange_0 (\bx,\by), \nabla_\by \lagrange_0 (\bx,\by) \right) = \left( \bA(\by-\by^0) - \bb, \bA(\bx-\bx^0) - \bb \right)$, and $\bz^0 +\left(\bx^\textrm{min},\bx^\textrm{min}\right)$ is a saddle point of $\lagrange_0$.

We follow the oracle-resisting proof strategy of \citet{nemirovsky1991optimality}, described as follows.
For each $i = 1,\dots,k$, we inductively define a \emph{rotated} biaffine function
\[
\lagrange_i (\bx^0,\by^0) = -\bb^\intercal (\bx-\bx^0) + (\bx-\bx^0)^\intercal \bA_i (\by-\by^0) - \bb^\intercal (\by-\by^0),
\]
where $\bA_i = \bU_i \bA \bU_i^\intercal$ for an orthogonal matrix $\bU_i \in \reals^{n\times n}$.
We will show that $U_i$ can be chosen to satisfy $\bU_i \bb = \bb$,
\begin{align}
\label{eqn:same_oracle_output}
\cO(\bz^j;\lagrange_i) = \cO(\bz^j;\lagrange_{i-1})
\end{align}
for $j=0,\dots,i-1$, and
\begin{align}
\label{eqn:induction_Krylov_plus_kernel}
\bx^j - \bx^0, \by^j - \by^0 \in \mathcal{K}_{j-1}(\bA_i;\bb) \oplus \bU_i \mathcal{N}_i = \bU_i \cK_{j-1}(\bA;\bb) \oplus \bU_i \cN_i
\end{align}
for $j=0,\dots,i$, where $\cN_i$ is a subspace of $\ker(\bA)$ such that $\dim(\cN_i) \le 2i$.
Note that \eqref{eqn:same_oracle_output} implies that the algorithm iterates $(\bz^j,\overline{\bz}^j)$ for $j=1,\dots,i$ do not change when $\lagrange_{i-1}$ is replaced by $\lagrange_i$.
Hence, this process sequentially adjusts the objective function $\lagrange$ upon observing an iterate $\bz^i$ to resist the algorithm from optimizing it efficiently.
Indeed, if \eqref{eqn:induction_Krylov_plus_kernel} holds with $i=j=k$, then
\begin{align*}
& \bx^k - \bx^0 = \bU_k q_\bx(\bA)\bb + \bU_k \bv_\bx^k \\
& \by^k - \by^0 = \bU_k q_\by(\bA)\bb + \bU_k \bv_\by^k
\end{align*}
for some polynomials $q_\bx, q_\by$ of degree $\le k-1$ and $\bv_\bx^k, \bv_\by^k \in \cN_i \subseteq \ker(\bA)$.
Thus
\begin{align*}
\nabla_\bx \lagrange_k (\bx^k,\by^k) = \bA_k (\by^k-\by^0) - \bb = \bU_k \bA \bU_k^\intercal \left( \bU_k q_y(\bA)\bb + \bU_k \bv_\by^k \right) - \bb = \bU_k \left( \bA q_y(\bA) - \bI \right) \bb
\end{align*}
and similarly
\[
\nabla_\by \lagrange_k (\bx^k,\by^k) = \bU_k \left( \bA q_x(\bA) - \bI \right) \bb,
\]
showing that
\[
\|\nabla \lagrange_k (\bz^k)\|^2 = \|\bU_k \left( \bA q_y(\bA) - \bI \right) \bb\|^2 + \|\bU_k \left( \bA q_x(\bA) - \bI \right) \bb\|^2 \ge \frac{2 \|\bx^\textrm{min}\|^2}{(2\lfloor k/2 \rfloor+1)^2}.
\]
Then the theorem statement follows from the fact that $\bz^\star = \bz^0 +(\bU_k \bx^{\textrm{min}}, \bU_k \bx^{\textrm{min}})$ is a saddle point of $\lagrange_k$.

It remains to provide an inductive scheme for choosing $\bU_i$.
We set $\bU_0 = \bI$ (so that $\bA_0 = \bA$), $\cN_0 = \{0\}$, and define $\cK_{-1}(\bA;\bb) = \{0\}$ for convenience.
Let $1\le i\le k$, and suppose that we already have an orthogonal matrix $\bU_{i-1}$ and $\cN_{i-1} \subseteq \ker(\bA)$ for which $\bU_{i-1} \bb = \bb$, $\dim (\cN_{i-1}) \le 2i-2$, and \eqref{eqn:induction_Krylov_plus_kernel} holds with $i-1$ (which is vacuously true when $i=1$).
Let
\[
(\bz^i, \overline{\bz}^i) = \cA_i \left( \bz^0, \cO(\bz^0;\lagrange_{i-1}), \dots, \cO(\bz^{i-1};\lagrange_{i-1}) \right).
\]
We want $\bU_i$ (to be defined) to satisfy $\bs_\bx^i, \bs_\by^i \in \bU_i \ker (\bA)$ while $\cK_{i-1}(\bA_{i-1};\bb) = \cK_{i-1}(\bA_i;\bb)$.
The latter condition is satisfied if $\bU_i = \bQ_i \bU_{i-1}$ for some orthogonal matrix $\bQ_i$ which preserves every element within
\[
\cJ_{i-1} = \cK_{i-1} (\bA_{i-1};\bb) \oplus \bU_{i-1} \cN_{i-1},
\]
because then it follows that $\bU_i \bb = \bQ_i \bU_{i-1} \bb = \bQ_i \bb = \bb$ and
\[
\cK_{i-1}(\bA_{i};\bb) = \bU_{i} \cK_{i-1} (\bA;\bb) = \bQ_i \bU_{i-1} \cK_{i-1} (\bA;\bb) = \bQ_i \cK_{i-1} (\bA_{i-1};\bb) = \cK_{i-1} (\bA_{i-1};\bb).
\]
Consider the decomposition
\begin{gather*}
\bx^i - \bx^0 = \Pi_{\cK_{i-1}(\bA_{i-1};\bb)} (\bx^i-\bx^0) + \bU_{i-1} \br^i_\bx + \bs^i_\bx\\
\by^i - \by^0 = \Pi_{\cK_{i-1}(\bA_{i-1};\bb)} (\by^i-\by^0) + \bU_{i-1} \br^i_\by + \bs^i_\by
\end{gather*}
where $\Pi$ denotes the orthogonal projection, $\br^i_\bx, \br^i_\by \in \cN_{i-1}$ and $\bs^i_\bx, \bs^i_\by \in \cJ_{i-1}^\perp$.
Since $\dim \ker (\bA) = n - 2m - 2 \ge n - k - 2$ and $\dim \left(\cN_{i-1}\right)^\perp \ge n - (2i-2) \ge n - 2k + 2$, we have
\[
\dim \left( \ker(\bA) \cap \left(\cN_{i-1}\right)^\perp \right) \ge n - 3k \ge 2,
\]
so there exist $\Tilde{\bs}^i_\bx, \Tilde{\bs}^i_\by \in \ker(\bA) \cap \left(\cN_{i-1}\right)^\perp$ such that $\|\Tilde{\bs}^i_\bx\| = \|\bs^i_\bx\|$, $\|\Tilde{\bs}^i_\by\| = \|\bs^i_\by\|$, and $\langle \Tilde{\bs}^i_\bx, \Tilde{\bs}^i_\by \rangle = \langle \bs^i_\bx, \bs^i_\by \rangle$.
Also, because $\ker(\bA) \perp \cK_{i-1}(\bA;\bb)$,
\[
\cJ_{i-1} = \bU_{i-1} \left( \cK_{i-1}(\bA;\bb) + \cN_{i-1} \right) \perp \bU_{i-1} \left( \ker(\bA) \cap \left(\cN_{i-1}\right)^\perp \right).
\]
This implies that there exists an orthogonal $\bQ_i \in \reals^{n \times n}$ satisfying
\begin{gather*}
\bQ_i\big|_{\cJ_{i-1}} = \mathrm{Id}_{\cJ_{i-1}} \\
\bQ_i \left(\bU_{i-1} \Tilde{\bs}^i_\bx\right) = \bs^i_\bx \\
\bQ_i \left(\bU_{i-1} \Tilde{\bs}^i_\by\right) = \bs^i_\by.
\end{gather*}
Now let $\bv^i_\bx = \br^i_\bx + \Tilde{\bs}^i_\bx \in \ker(\bA), \bv^i_\by = \br^i_\by + \Tilde{\bs}^i_\by \in \ker(\bA)$, and
\begin{gather*}
\bU_i \stackrel{\Delta}{=} \bQ_i \bU_{i-1} \\
\cN_i \stackrel{\Delta}{=} \cN_{i-1} + \spann\{\bv^i_\bx, \bv^i_\by\}.
\end{gather*}
Then clearly $\bU_i \bb = \bb$, $\cN_i \subseteq \ker(\bA)$, and $\dim \cN_i \le 2i$.
Next, for each $j=0,\dots,i-1$, we have
\[
\bx^j - \bx^0, \by^j - \by^0 \in \cK_{j-1}(\bA_{i-1};\bb) \oplus \bU_{i-1} \cN_{i-1} \subseteq \cK_{j-1} (\bA_i;\bb) \oplus \bU_i \cN_i
\]
since $\bQ_i$ preserves $\cJ_{i-1}$ and $\cN_{i-1} \subseteq \cN_i$.
Moreover, because $\bU_{i-1} \br^i_\bx = \bQ_i \bU_{i-1} \br^i_\bx = \bU_i \br^i_\bx$ and $\bs^i_\bx = \bQ_i \bU_{i-1} \Tilde{\bs}^i_\bx = \bU_i \Tilde{\bs}^i_\bx$,
\[
\bx^i - \bx^0 = \Pi_{\cK_{i-1}(\bA_{i-1};\bb)} (\bx^i-\bx^0) + \bU_i (\br^i_\bx + \Tilde{\bs}^i_\bx) \in \cK_{i-1}(\bA_{i-1};\bb) \oplus \bU_i \cN_i = \cK_{i-1} (\bA_{i};\bb) \oplus \bU_i \cN_i
\]
and similarly $\by^i - \by^0 \in \cK_{i-1} (\bA_{i};\bb) \oplus \bU_i \cN_i$.
This proves \eqref{eqn:induction_Krylov_plus_kernel}.

Finally, for $j=0,\dots,i-1$,
\begin{align*}
\nabla_\bx \lagrange_i (\bx^j, \by^j) = \bA_i (\by^j - \by^0) - \bb = \bQ_i \bA_{i-1} \bQ_i^\intercal (\by^j - \by^0) - \bb.
\end{align*}
But $\bQ_i^\intercal (\by^j - \by^0) = \by^j - \by^0$ because $\by^j - \by^0 \in \cK_{j-1}(\bA_{i-1};\bb) \oplus \bU_{i-1} \cN_{i-1} \subseteq \cJ_{i-1}$, and
\begin{align*}
\bA_{i-1} (\by^j - \by^0) \in \bA_{i-1} \cK_{j-1} (\bA_{i-1};\bb) \oplus \bA_{i-1} \bU_{i-1} \cN_{i-1} = \cK_j (\bA_{i-1};\bb) \subseteq \cJ_{i-1},
\end{align*}
which shows that $\nabla_\bx \lagrange_i (\bx^j, \by^j) = \bQ_i \bA_{i-1} \bQ_i^\intercal (\by^j - \by^0) - \bb = \bA_{i-1} (\by^j - \by^0) - \bb = \nabla_\bx \lagrange_{i-1} (\bx^j, \by^j)$.
Arguing analogously for the $\by$-variable gives $\nabla_\by \lagrange_i (\bx^j, \by^j) = \nabla_\by \lagrange_{i-1} (\bx^j, \by^j)$, proving \eqref{eqn:same_oracle_output}.
This completes the induction step, and hence the proof.
\end{proof}

\section{Experimental details}
\label{section:experimental_details}

\subsection{Exact forms of the construction from \citet{ouyang2019lower}}

Following \citet{ouyang2019lower}, we use
\[
\bA = \frac{1}{4} \begin{bmatrix}
  &  &  & -1 & 1 \\
  &  & \iddots & \iddots & \\
  & -1 & 1 \\
  -1 & 1\\
  1
\end{bmatrix} \in \reals^{n\times n}, \quad
\bb = \frac{1}{4} \begin{bmatrix}
1 \\ 1 \\ \vdots \\ 1 \\ 1 
\end{bmatrix} \in \reals^n, \quad
\bh = \frac{1}{4} \begin{bmatrix}
0 \\ 0 \\ \vdots \\ 0 \\ 1
\end{bmatrix} \in \reals^n,
\]
and $\bH = 2\bA^\intercal \bA$.
\citet{ouyang2019lower} shows that $\|\bA\| \le \frac{1}{2}$, which implies $\|\bH\| \le \frac{1}{2}$.
Therefore \eqref{eqn:L2} is a $1$-smooth saddle function.

\subsection{Best-iterate gradient norm bound for EG}

In Figure \ref{fig:loglogs}, we indicated theoretical upper bounds for EG.
To clarify, there is no known last-iterate convergence result for EG with respect to $\|\sop(\cdot)\|^2$.
However, it is straightforward to derive $\mathcal{O}(R^2/k)$ \emph{best-iterate} convergence via standard summability arguments in weak convergence proofs for EG.
Although there is no theoretical guarantee that $\|\sop(\bz^k)\|^2$ will monotonically decrease with EG, in our experiments on both examples, they did monotonically decrease (see Figures \ref{fig:two_dim_loglog}, \ref{fig:ouyang}).
Therefore, we safely used the best-iterate bounds to visualize the upper bound for EG in Figure \ref{fig:loglogs}.
For the sake of completeness, we derive the best-iterate bound below.

\begin{lemma}
\label{lemma:summable_term}
Let $\lagrange\colon \reals^n \times \reals^m \to \reals$ be an $R$-smooth convex-concave saddle function with a saddle point $\bz^\star$.
Let $\bz \in \reals^n \times \reals^m$ and $\alpha \in \left(0,\frac{1}{R}\right)$.
Then $\bw = \bz - \alpha\sop(\bz)$ and $\bz^+ = \bz - \alpha\sop(\bw)$ satisfy
\[
\|\bz-\bz^\star\|^2 - \|\bz^+-\bz^\star\|^2 \ge (1-\alpha^2 R^2) \|\bz-\bw\|^2.
\]
\end{lemma}

\begin{proof}
\begin{align*}
\|\bz-\bz^\star\|^2 - \|\bz^+-\bz^\star\|^2 & = \left( \|\bz-\bw\|^2 + 2 \langle \bz-\bw, \bw-\bz^\star \rangle + \|\bw-\bz^\star\|^2 \right) \\
& \quad \quad - \left( \|\bz^+-\bw\|^2 + 2 \langle \bz^+-\bw, \bw-\bz^\star \rangle + \|\bw-\bz^\star\|^2 \right) \\
& = \|\bz-\bw\|^2 - \|\bz^+-\bw\|^2 + 2\langle \bz-\bz^+, \bw-\bz^\star\rangle \\
& \ge \|\bz-\bw\|^2 - \|\bz^+-\bw\|^2.
\end{align*}
The last inequality is just monotonicity: $\langle \bz-\bz^+, \bw-\bz^\star \rangle = \alpha \langle \sop(\bw), \bw-\bz^\star \rangle \ge 0$.
Now the conclusion follows from
\begin{align*}
\|\bz^+-\bw\|^2 = \left\| (\bz-\alpha\sop(\bw)) - (\bz-\alpha\sop(\bz)) \right\|^2 = \alpha^2 \|\sop(\bz) - \sop(\bw)\|^2 \le \alpha^2 R^2 \|\bz-\bw\|^2,
\end{align*}
where the last inequality follows from $R$-Lipschitzness of $\sop$.
\end{proof}

Now fix an integer $k\ge 0$, and consider the EG iterations
\begin{align*}
    \bz^{i+1/2} & = \bz^i - \alpha \sop (\bz^i) \\
    \bz^{i+1} & = \bz^i - \alpha \sop (\bz^{i+1/2})
\end{align*}
for $i=0,\dots,k$.
Applying Lemma~\ref{lemma:summable_term} with $\bz=\bz^i$, $\bw=\bz^{i+1/2}$ and $\bz^+ = \bz^{i+1}$, we have
\begin{align}
\label{eqn:summable_term_EG}
\|\bz^i - \bz^\star\|^2 - \|\bz^{i+1}-\bz^\star\|^2 \ge (1-\alpha^2 R^2) \|\bz^i - \bz^{i+1/2}\|^2 = (1-\alpha^2 R^2) \alpha^2 \|\sop(\bz^i)\|^2
\end{align}
for $i=0,\dots,k$.
Summing up the inequalities \eqref{eqn:summable_term_EG} for all $i=0,\dots,k$, we obtain
\begin{align*}
\|\bz^0 - \bz^\star\|^2 - \|\bz^{k+1} - \bz^\star\|^2
\ge (1-\alpha^2 R^2)\alpha^2 \sum_{i=0}^k \|\sop(\bz^i)\|^2.
\end{align*}
The left hand side is at most $\|\bz^0-\bz^\star\|^2$, while the right hand side is lower bounded by
\begin{align*}
(1-\alpha^2 R^2)\alpha^2 \, (k+1) \min_{i=0,\dots,k} \|\sop(\bz^i)\|^2.
\end{align*}
Therefore we conclude that
\[
\min_{i=0,\dots,k} \|\sop(\bz^i)\|^2 \le \frac{C\|\bz^0-\bz^\star\|^2}{k+1}
\]
where $C = \frac{1}{\alpha^2 (1-\alpha^2 R^2)}$.

\subsection{ODE flows for $\boldsymbol{\lagrange(x,y) = xy}$}

Interestingly, the continuous-time flows with $\lagrange(x,y) = xy$ have exact closed-form solutions.

Note that $\sop(x,y) = \begin{bmatrix}
0 & 1 \\ -1 & 0 \end{bmatrix} \begin{bmatrix} x \\ y \end{bmatrix}$.
Therefore,
\begin{align*}
    \sop_\lambda (x,y) = \frac{1}{\lambda} \left(\begin{bmatrix}
    1 & 0 \\ 0 & 1 \end{bmatrix} - \begin{bmatrix}
    1 & \lambda \\ -\lambda & 1 \end{bmatrix}^{-1} \right) \begin{bmatrix}
    x \\ y
    \end{bmatrix}
    = \begin{bmatrix}
    \frac{\lambda}{1+\lambda^2} & \frac{1}{1+\lambda^2} \\
    -\frac{1}{1+\lambda^2} & \frac{\lambda}{1+\lambda^2}
    \end{bmatrix} \begin{bmatrix}
    x \\ y
    \end{bmatrix}.
\end{align*}
The solution to the Moreau--Yosida regularized flow
\begin{align*}
    \begin{bmatrix}
    \dot{x} \\ \dot{y}
    \end{bmatrix} = \begin{bmatrix}
    -\frac{\lambda}{1+\lambda^2} & -\frac{1}{1+\lambda^2} \\
    \frac{1}{1+\lambda^2} & -\frac{\lambda}{1+\lambda^2}
    \end{bmatrix} \begin{bmatrix}
    x \\ y
    \end{bmatrix}
\end{align*}
can be obtained with the matrix exponent. The results are
\begin{align*}
    & x(t) = \exp\left( -\frac{\lambda}{1+\lambda^2} t\right) \left( x^0 \cos \frac{t}{1+\lambda^2} - y^0 \sin \frac{t}{1+\lambda^2} \right) \\
    & y(t) = \exp\left( -\frac{\lambda}{1+\lambda^2} t\right) \left( y^0 \cos \frac{t}{1+\lambda^2} + x^0 \sin \frac{t}{1+\lambda^2} \right).
\end{align*}

The anchored flow ODE for $\lagrange(x,y) = xy$ is given by
\begin{align*}
    & \dot{x}(t) = -y(t) + \frac{1}{t}\left(x^0 - x(t)\right)\\
    & \dot{y}(t) = x(t) + \frac{1}{t}\left(y^0 - y(t)\right).
\end{align*}
From the first equation, we have $\frac{d}{dt}(tx(t)) = t\dot{x}(t) + x(t) = -ty(t) + x^0$, while similar manipulation of the second equation gives $\frac{d}{dt}(ty(t)) = tx(t) + y^0$.
Therefore,
\begin{align*}
    & \frac{d^2}{dt^2}(tx(t)) = - \frac{d}{dt}(ty(t)) = -tx(t) - y^0 \\
    & \frac{d^2}{dt^2}(ty(t)) = \frac{d}{dt}(tx(t)) = -ty(t) + x^0,
\end{align*}
which gives
\begin{align*}
    & tx(t) = c_1 \cos t - c_2 \sin t - y^0 \\
    & ty(t) = c_1 \sin t + c_2 \cos t + x^0.
\end{align*}
Using the initial conditions to determine the coefficients $c_1, c_2$, we obtain
\begin{align*}
    & x(t) = \frac{y^0 \cos t + x^0 \sin t - y^0}{t} \\
    & y(t) = \frac{y^0 \sin t - x^0 \cos t + x^0}{t}.
\end{align*}

\section{Connection to CLI lower bounds}
\label{section:CLI}
In this section, we discuss how EAG relates to the prior work on complexity lower bounds on the class of CLI and SCLI algorithms, introduced and studied in \cite{arjevani2015lower, arjevani2016iteration, azizian2020tight, golowich2020last}.
Specifically, we show that EAG is not SCLI, so it can break the $\Omega(R^2/k)$ lower bound on squared gradient norm for the 1-SCLI class derived by \citet{golowich2020last}.
On the other hand, we show that EAG is 2-CLI in the sense of \citet{golowich2020last}, and that EAG belongs to an extended class of 1-CLI algorithms. 

\subsection{Lower bounds for 1-SCLI and non-stationarity of EAG }
\label{section:1-SCLI_lower_bound}
We start with the notion of 1-SCLI algorithms by \citet{golowich2020last}.
Consider an algorithm $\cA$ for finding saddle points of biaffine functions of the form
\[
\lagrange(\bx,\by) = \bb^\intercal \bx + \bx^\intercal \bA \by - \bc^\intercal \by,
\]
where $(\bx,\by)\in \reals^n \times \reals^n$.
We say $\cA$ is \emph{1-stationary canonical linear iterative (1-SCLI)} if there exist some fixed matrix mappings $\bC, \bN: \reals^{2n \times 2n} \to \reals^{2n \times 2n}$ such that 
\begin{align}
    \label{eqn:1-SCLI}
    \bz^{k+1} = \bC\left(\begin{bmatrix} \bO & \bA \\ -\bA^\intercal & \bO \end{bmatrix}\right) \bz^k + \bN\left(\begin{bmatrix} \bO & \bA \\ -\bA^\intercal & \bO \end{bmatrix}\right) \begin{bmatrix} \bb \\ \bc \end{bmatrix}
    = \bC(\bB) \bz^k + \bN(\bB) \bv
\end{align}
for $k\ge 0$, where
\begin{align*}
    \bB = \begin{bmatrix} \bO & \bA \\ -\bA^\intercal & \bO \end{bmatrix} \in \reals^{2n \times 2n}, \quad \bv = \begin{bmatrix} \bb \\ \bc \end{bmatrix} \in \reals^{2n}.
\end{align*}
Following the convention of \citet{azizian2020tight} and \citet{golowich2020last}, we also require that $\bC, \bN$ are matrix polynomials.
The classical extragradient method (EG) is an 1-SCLI algorithm: with $\sop(\bz) = \bB\bz + \bv$, we can express EG as
\begin{align*}
    \bz^{k+1} & = \bz^k - \alpha \sop\left(\bz^k - \alpha \sop(\bz^k)\right) \\
    & = \bz^k - \alpha \left( \bB\left(\bz^k - \alpha \bB\bz^k - \alpha \bv\right) + \bv \right) \\
    & = \left( \bI - \alpha \bB + \alpha^2 \bB^2 \right) \bz^k - \alpha (\bI-\alpha \bB) \bv,
\end{align*}
which is of the 1-SCLI form.

A $1$-SCLI algorithm $\cA$ is \emph{consistent} with respect to an invertible matrix $\bB$ if for any $\bv \in \reals^{2n}$, iterates $\{\bz^k\}_{k\ge 0}$ produced by  $\cA$
satisfy
\[
\bz^k \to \bz^\star = -\bB^{-1}\bv.
\]
If $\cA$ is consistent with respect to $\bB$, then for any $\bw=\bB^{-1} \bv\in \reals^{2n}$, we have
\begin{align*}
    -\bw = -\bB^{-1} \bv &= \lim_{k\to\infty} \bz^{k+1} \\
    &= \lim_{k\to\infty} \bC(\bB)\bz^k + \bN(\bB)\bv \\
    &= \bC(\bB) (-\bB^{-1} \bv) + \bN(\bB) \bv \\
    &= \left( -\bC(\bB) + \bN(\bB) \bB \right) \bw.
\end{align*}
As this holds for all $\bw\in \reals^{2n}$, we have the following result.
\begin{lemma}[\citet{arjevani2015lower}]
If a 1-SCLI algorithm $\cA$ described by \eqref{eqn:1-SCLI} is consistent with respect to $\bB$, then
\begin{align}
    \label{eqn:consistency}
    \bI + \bN(\bB) \bB = \bC(\bB).
\end{align}
\end{lemma}
Indeed, the 1-SCLI formulation of EG satisfies \eqref{eqn:consistency}.

For the class of consistent 1-SCLI algorithms, \citet{golowich2020last} established $\Omega(1/k)$ a complexity lower bound on squared gradient norm.
\begin{theorem}[\citet{golowich2020last}]
Let $k\ge 0$ and $n \ge 1$. Then for any consistent 1-SCLI algorithm of the form \eqref{eqn:1-SCLI} with $\deg \bN = d_\bN$, there exist a biaffine function $\lagrange (\bx,\by) = \bb^\intercal \bx + \bx^\intercal \bA \by - \bc^\intercal \by$ on $\reals^n \times \reals^n$ with invertible $\bA$, for which
\begin{align*}
    \|\nabla\lagrange(\bz^k)\|^2 \ge \frac{R^2 \|\bz^0 - \bz^\star\|^2}{20(d_\bN+1)^2 k} = \Omega \left( \frac{R^2 \|\bz^0-\bz^\star\|^2}{k} \right),
\end{align*}
where $\bz^\star$ is the unique saddle point of $\lagrange$.
\end{theorem}
To clarify, $\deg \bN$ refers to the degree of the matrix polynomial defining $\bN$.
1-SCLI algorithms with $d_\bC = \deg \bC = 1$ forms a subclass of $\fA_{\textrm{sim}}$ and $\fA_{\textrm{sep}}$.
(Even if $d_\bC>1$, one can still view 1-SCLI algorithms as instances of $\fA_{\textrm{sim}}$ or $\fA_{\textrm{sep}}$ by introducing $d_\bC-1$ dummy iterates for each 1-SCLI iteration.)
However, EAG is an algorithm that belongs to $\fA_{\textrm{sim}}$ but is not 1-SCLI; if it was, a contradiction would occur, as $\|\nabla\lagrange (\bz^k)\|^2 \le \mathcal{O}(1/k^2)$ for EAG.
In fact, it is intuitively clear that EAG is not 1-SCLI; the S in 1-\textbf{S}CLI stands for stationary, but EAG has anchoring coefficients $\frac{1}{k+2}$ that vary over iterations.


\subsection{Understanding EAG as a CLI algorithm}
In this section, we show that EAG algorithms are (non-stationary) 2-CLI, and that we can expand the definition of 1-CLI algorithms to accommodate EAG.


First, we state the definition of $m$-CLI algorithms introduced by \citet{arjevani2016iteration} adapted to the case of biaffine saddle functions.
For $m \ge 1$, an $m$-CLI algorithm $\cA$ takes $m$ initial points $\bz^0_1, \dots, \bz^0_m$ and at each iteration $k \ge 0$, outputs
\begin{align}
\label{eqn:m-CLI}
    \bz^{k+1}_i = \sum_{j=1}^m \bC_{ij}^{(k)}(\bB) \, \bz^k_j + \bN_i^{(k)}(\bB) \, \bv
\end{align}
for $i=1,\dots,m$, where $\bC_{ij}^{(k)}, \bN_i^{(k)}: \reals^{2n\times 2n} \to \reals^{2n \times 2n}$ for $i,j=1,\dots,m$ are matrix polynomials that depend on $k$ but not on $\{\bz^k_1, \dots, \bz^k_m\}_{k\ge 0}$.
In the case where $\bC_{ij}^{(k)} \equiv \bC_{ij}$ and $\bN_i^{(k)} \equiv \bN_i$ for all $i,j=1,\dots,m$ and $k\ge 0$, we say $\cA$ is stationary.
Indeed, when $m=1$, this definition of stationary 1-CLI coincides with that of 1-SCLI given in Section \ref{section:1-SCLI_lower_bound}.
Also note that the definition \eqref{eqn:m-CLI} includes algorithms that obtain $\bz^{k+1}$ with $m$ previous iterates $\bz^k, \bz^{k-1}, \dots, \bz^{k-m+1}$, by letting $\bz^{k}_i = \bz^{k+1-i}$ for $i=1,\dots,m$.

\begin{center}
\begin{table}[h]
\centering
\begin{tabular}{@{}ccccc@{}}
\toprule
Performance measure & Algorithm class & Lower bound & Best known rate & Order-optimality \\ \midrule
 & 1-SCLI & \makecell{$\Omega\left( \frac{R}{\sqrt{k}} \right)$ \\ \citep{golowich2020last}} & \makecell{$\cO\left( \frac{R}{\sqrt{k}} \right)$ \\ \citep{golowich2020last}*} & Established* \\ \cmidrule(l){2-5} 
\makecell{Duality gap\\(Last iterate)} & 1-CLI & \makecell{$\Omega\left( \frac{R}{k} \right)$ \\ (\citet{nemirovsky1992information}, \\ \citet{nemirovski2004prox})} & \makecell{$\cO\left( \frac{R}{\sqrt{k}} \right)$ \\ \citep{golowich2020last}*} & Unknown \\ \cmidrule(l){2-5}
 & \makecell{$m$-CLI \\ ($m \ge 2$)} & \makecell{$\Omega\left( \frac{R}{k} \right)$ \\ (\citet{nemirovsky1992information}, \\ \citet{nemirovski2004prox})} & \makecell{$\cO\left( \frac{R}{k} \right)$ \\ (\citet{nemirovski2004prox}, \\ \citet{golowich2020last})} & Established \\ \midrule
 & 1-SCLI & \makecell{$\Omega\left( \frac{R^2}{k} \right)$ \\ \citep{golowich2020last}} & \makecell{$\cO\left( \frac{R^2}{k} \right)$ \\ \citep{golowich2020last}*} & Established* \\ \cmidrule(l){2-5}
\multirow{3}{*}{\makecell{\\Squared gradient norm\\(Last iterate)}} & 1-CLI & \makecell{$\Omega\left( \frac{R^2}{k^2} \right)$ \\ \citep{nemirovsky1992information}} & \makecell{$\cO\left( \frac{R^2}{k} \right)$ \\ \citep{golowich2020last}*} & Unknown \\ \cmidrule(l){2-5}
 & \makecell{Translated\\1-CLI} & \makecell{$\Omega\left( \frac{R^2}{k^2} \right)$ \\ \citep{nemirovsky1992information}} & \makecell{$\cO\left( \frac{R^2}{k^2} \right)$ \\ (This paper)} & Established \\ \cmidrule(l){2-5}
 & \makecell{$m$-CLI \\ ($m \ge 2$)} & \makecell{$\Omega\left( \frac{R^2}{k^2} \right)$ \\ \citep{nemirovsky1992information}} & \makecell{$\cO\left( \frac{R^2}{k^2} \right)$ \\ (This paper)} & Established \\ \bottomrule
\end{tabular}
\caption{Lower bounds and best known rates for CLI algorithm classes
(* means that the result holds with the additional assumption that the derivative of $\bG$ is Lipschitz continuous).
}
\label{tab:cli_lower_bounds}
\end{table}
\end{center}

\vspace{-.5cm}

\citet{golowich2020last} showed that the averaged EG iterates, which have rate $\mathcal{O}(1/k)$ on duality gap, can be written in 2-CLI form; hence, the $\Omega(1/\sqrt{k})$ 1-SCLI lower bound on duality gap therein cannot be generalized to $m$-CLI algorithms for $m\ge 2$.
They then posed the open problem of whether the $\Omega(1/\sqrt{k})$ 1-SCLI lower bound on duality gap can be generalized to 1-CLI algorithms.
Below, we provide a similar discussion on rates on squared gradient norm.

It is straightforward to see that EAG is 2-CLI; define $\bz^{k+1}_2 = \bz^{k}_2 = \cdots = \bz^0_2 = \bz^0 = \bz^0_1$ for all $k\ge 0$, and
\begin{align}
    \nonumber
    \bz^{k+1}_1 & = \bz^k_1 - \alpha_k \sop \left( \bz^k_1 - \alpha_k \sop(\bz^k_1) + \frac{1}{k+2}(\bz^0-\bz^k_1) \right) + \frac{1}{k+2} (\bz^0-\bz^k_1) \\
    \label{eqn:EAG-2-CLI}
    & = \left( \frac{k+1}{k+2} \bI - \frac{k+1}{k+2} \alpha_k \bB + \alpha_k^2 \bB^2 \right) \bz^k_1 + \frac{1}{k+2} (\bI-\alpha_k \bB) \bz^k_2 - \alpha_k (\bI-\alpha_k \bB) \bv.
\end{align}
For EAG-C, one can alternatively eliminate the dependency on $\bz^0$ to define $\bz^{k+1}$ in terms of $\bz^k$, $\bz^{k-1}$, and $\bv$; respectively multiply $(k+2)$ and $(k+1)$ to the following identities 
\begin{gather*}
    \bz^{k+1} = \left( \frac{k+1}{k+2} \bI - \frac{k+1}{k+2} \alpha\bB + \alpha^2 \bB^2 \right) \bz^k + \frac{1}{k+2} (\bI-\alpha \bB) \bz^0 - \alpha (\bI-\alpha \bB) \bv \\
    \bz^k = \left( \frac{k}{k+1} \bI - \frac{k}{k+1} \alpha\bB + \alpha^2 \bB^2 \right) \bz^{k-1} + \frac{1}{k+1} (\bI-\alpha \bB) \bz^0 - \alpha (\bI-\alpha \bB) \bv
\end{gather*}
and subtract to eliminate $\bz^0$.
Since EAG has $\mathcal{O}(1/k^2)$ rate, this reformulation shows that the $\Theta(1/k)$ 1-SCLI lower bound on the squared gradient norm cannot be generalized to 2-CLI algorithms.

Furthermore, EAG also provides a partial resolution, in the negative, of the open problem of whether the $\Theta(1/k)$ 1-SCLI lower bound on the squared gradient norm can be generalized to 1-CLI algorithms.
Observe that if we translate the given problem to set $\bz^0 = 0$, keeping the sequence $\bz^k_2$ is no longer necessary, and \eqref{eqn:EAG-2-CLI} reduces to 1-CLI form.
Such translation is not allowed in the definition \eqref{eqn:m-CLI}, but it is reasonable to consider an expanded class of algorithms that are $1$-CLI up to translation.
Precisely, define an algorithm $\cA$ to be \emph{translated 1-CLI} if it takes the form
\begin{align*}
    \bz^{k+1} = \bC^{(k)}(\bB)(\bz^k) + \bN^{(k)}(\bB)(\bv)
\end{align*}
when $\bz^0 = 0$, and is \emph{translation invariant} in the sense that
\begin{align*}
    \bz^k = \cA(\bz^0, \bz^1, \dots, \bz^{k-1}; \bL) = \bz^0 + \cA(0, \bz^1 - \bz^0, \dots, \bz^{k-1}-\bz^0; \bL_{\bz^0})
\end{align*}
when $\bz^0 \ne 0$, where $\bL_{\bz^0}(\bx,\by) = \bL(\bx+\bx^0, \by+\by^0)$.
That is, the iterates of $\cA$ are generated equivalently by starting with $\bz^0 = 0$ and applying $\cA$ to the translated objective $\bL_{\bz^0}$. 
The concept of translated 1-CLI can be viewed as a generalization of consistent 1-SCLI algorithms;
observe that we can rewrite \eqref{eqn:1-SCLI} as
\begin{align*}
    \bz^{k+1} - \bz^0 = \bC(\bB)(\bz^k - \bz^0) + \bN(\bB) (\bB \bz^0 + \bv) - (\bI + \bN(\bB)\bB - \bC(\bB)) \bz^0,
\end{align*}
which shows that a 1-SCLI algorithm is translation invariant if and only if it satisfies the consistency formula \eqref{eqn:consistency}.
Since EAG has $\mathcal{O}(1/k^2)$ rate and is a translated 1-CLI algorithm, our results prove that the $\Theta(1/k)$ 1-SCLI lower bound on the squared gradient norm can be generalized to translated 1-CLI algorithms.

\end{document}